\documentclass[a4paper,10pt]{article}

\usepackage[english]{babel}
\usepackage{graphicx}
\usepackage{url}
\usepackage{booktabs}
\usepackage{array}
\input xy
\xyoption{all}
\usepackage{comment,xcomment,amssymb,amsmath,amsthm,color}
\usepackage{rotating}
\usepackage{latexsym,mathtools}
\usepackage{caption}
\newcommand{\abs}[1]{\left\vert#1\right\vert}
\newcommand{\R}{\mathbb{R}}
\newcommand{\im}{\mathrm{Im}\,}         
\newcommand{\lie}[1]{\mathfrak{#1}}     
\newcommand{\Z}{\mathbb{Z}}

\newcommand{\N}{\mathbb{N}}

\newcommand{\hook}{\lrcorner\,}

\newcommand{\LieG}[1]{\mathrm{#1}}      

\newcommand{\SU}{\mathrm{SU}}
\newcommand{\Sp}{\mathrm{Sp}}

\newcommand{\SO}{\mathrm{SO}}

\newcommand{\Gtwo}{\mathrm{G}_2}

\newcommand{\so}{\mathfrak{so}}

\newcommand{\GL}{\mathrm{GL}}
\newcommand{\SL}{\mathrm{SL}}

\newcommand{\id}{\mathrm{Id}}   

\newcommand{\ann}{^\mathrm{o}}
\newcommand{\gl}{\lie{gl}}
\newcommand{\Sl}{\lie{sl}}
\newcommand{\Span}[1]{\operatorname{Span}\left\{#1\right\}}
\newcommand{\tran}[1]{\prescript{t}{}{#1}}

\DeclareMathOperator{\ad}{ad}

\DeclareMathOperator{\coker}{Coker}
\DeclareMathOperator{\hol}{Hol}

\DeclareMathOperator{\ric}{Ric}

\theoremstyle{plain}
\newtheorem{proposition}{Proposition}
\newtheorem{theorem}[proposition]{Theorem}
\newtheorem{lemma}[proposition]{Lemma}
\newtheorem{corollary}[proposition]{Corollary}

\theoremstyle{definition}

\newtheorem{example}[proposition]{Example}
\theoremstyle{remark}
\newtheorem{remark}{Remark}

\newcommand{\D}[1]{\frac{\partial}{\partial #1}}
\title{The Ricci tensor of almost parahermitian manifolds}
\author{Diego Conti and Federico A. Rossi}

\begin{document}

\maketitle

\begin{abstract}
We study the pseudoriemannian geometry of almost parahermitian manifolds, obtaining a formula for the Ricci tensor of the Levi-Civita connection. The formula uses the intrinsic torsion of an underlying $\SL(n,\R)$-structure; we express it in terms of exterior derivatives of some appropriately defined differential forms.

As an application, we construct Einstein and Ricci-flat examples on Lie groups. We disprove the parak\"ahler version of the Goldberg conjecture, and obtain the first compact examples of a non-flat, Ricci-flat nearly parak\"ahler structure.

We study the paracomplex analogue of the first Chern class in complex geometry, which obstructs the existence of Ricci-flat parak\"ahler metrics.
\end{abstract}

\renewcommand{\thefootnote}{\fnsymbol{footnote}} 
\footnotetext{\emph{MSC class}: 53C15; 53C10, 53C29, 53C50.}
\footnotetext{This work was partially supported by FIRB 2012 ``Geometria differenziale e teoria geometrica delle funzioni'' and GNSAGA of INdAM.}
\renewcommand{\thefootnote}{\arabic{footnote}} 

\section{Introduction}
\label{intro}
Paracomplex geometry was introduced by Libermann \cite{Libermann:paracomplex1952}; in analogy with complex geometry, it is defined by a tensor $K$ with $K^2=\id$, whose eigenspaces are integrable distributions of dimension $n$. The local geometry is that of a product, but things become more complicated if a metric enters the picture. The natural compatibility condition to impose is that $K$ be an anti-isometry, so that $F=g(K\cdot, \cdot)$ defines a two-form; if this form is closed, the metric is said to be parak\"ahler. Such a metric is necessarily of neutral signature; its holonomy is contained in $\GL(n,\R)$, endowed with its standard action on $\R^n\oplus(\R^n)^*$. 

Parak\"ahler manifolds carry a natural bilagrangian structure; as such, they form a natural object of study in symplectic geometry (see \cite{EtayoSantamaria,HarveyLawson}); they also provide a natural setting for the study of the mean curvature flow, which is proved to preserve Lagrangian submanifolds in \cite{ChursinSchaferSmoczyk}, under some  assumptions on the curvature. Parak\"ahler geometry  also finds applications in physics (\cite{CMMS:VectorI2004}) and in the study of optimal transport (\cite{KimMcCannWarren}). We refer to \cite{CruceanuFortunyGadea:survey1996} for a survey; a more recent reference is  \cite{AlekseevskyMedoriTomassini:Homogeneous2009}.

Like in K\"ahler geometry, the Ricci tensor of a parak\"ahler manifold is given by $-\frac12 dd^c\log \abs{\phi}^2$, where $\phi$ is a local holomorphic volume form (see \cite{HarveyLawson}); this indicates that the existence of a parallel paracomplex volume form forces the Ricci to be zero. In terms of the restricted holonomy $\hol_0$, this condition is equivalent to $\hol_0\subset\SL(n,\R)$; thus, pseudoriemaniann metrics of neutral signature with (restricted) holonomy contained in $\SL(n,\R)$ are Ricci-flat.

This observation leads us to consider the smaller structure group $\SL(n,\R)$ rather than $\GL(n,\R)$; we do not impose integrability conditions, so that the eigendistributions of $K$ are not necessarily integrable, nor is $F$ necessarily closed. The failure of the holonomy condition is measured by a tensor called \emph{intrinsic torsion}; it follows from the above that the intrinsic torsion determines the Ricci tensor. An explicit formula to this effect is the main result of this paper (Theorems~\ref{thm:ricci}, \ref{thm:RicciS2}). We note that similar situations have been studied in \cite{Bryant:remarks} and \cite{BedulliVezzoni:SU3} for the Riemannian holonomy groups $\Gtwo$ and $\SU(3)$; see also \cite{CabreraSwann} for similar computations relative to the group $\Sp(n)\Sp(1)$. Our methods, however, are more akin to those of \cite{Conti:qc}, and they allow us to find an explicit formula valid in any dimension.

Intrinsic torsion relative to the structure group $\GL(n,\R)$ can be identified with the covariant derivative of $F$ under the Levi-Civita connection; the results of \cite{GadeaMasque} imply that, for $n\geq 3$, $\GL(n,\R)$-intrinsic torsion decomposes into eight components, corresponding to eight non-isomorphic $\GL(n,\R)$-modules $W_1\oplus\dotsb\oplus W_8$. Under $\SL(n,\R)$, these components remain irreducible, but two extra components appear, giving rise to ten intrinsic torsion classes. This situation is somewhat different from that of almost hermitian  geometry, where in six dimensions the Gray-Hervella intrinsic torsion classes become reducible upon reducing from $\LieG{U}(3)$ to $\SU(3)$ (see \cite{GrayHervella,ChiossiSalamon}).

Beside the fundamental form $F$, a manifold with an $\SL(n,\R)$-structure carries decomposable $n$-forms $\alpha,\beta$, characterized as volume forms on the two distributions defined by $K$. In fact, the structure group $\SL(n,\R)$ is the largest subgroup of $\GL(2n,\R)$ that fixes the corresponding elements $F$, $\alpha$ and $\beta$ of $\Lambda(\R^{2n})$. These differential forms are closed if and only if the intrinsic torsion is zero (see Proposition~\ref{prop:forms}); structure groups satisfying this condition are known as strongly admissible (see \cite{Bryant}). This allows us to restate our formula for the Ricci tensor purely in terms of exterior derivatives (Theorem~\ref{thm:restateRicViaForms}).

The language of $\SL(n,\R)$-structure enables us to construct a cohomological invariant of paracomplex manifolds, analogous to the first Chern class in complex geometry, which obstructs the existence of Ricci-flat parak\"ahler metrics; we obtain a sufficient topological condition for its vanishing (Theorem~\ref{thm:aubinyau_dei_poveri}). One difference with the K\"ahler case is that compact parak\"ahler Einstein manifolds are necessarily Ricci-flat, as we show in Proposition~\ref{prop:pkeinsteincompact}. 
Thus, unlike in the compact K\"ahler case, the invariant does not quite describe whether compatible parak\"ahler Einstein metrics have zero or non-zero scalar curvature. This is also not true in the non-compact setting, as we show by producing a non-compact parak\"ahler Einstein manifold with nonzero scalar curvature where the invariant is zero (Example~\ref{ex:Einstein}).

We then turn to the construction of Einstein examples. Examples of compact, Ricci-flat parak\"ahler manifolds appear in \cite{HarveyLawson}; these examples, however, are also flat. Non-flat examples on compact nilmanifolds were constructed by the second author in \cite{RossiPhD}; a similar example appears in Section~\ref{sec:examples}. These examples are homogeneous, marking a difference between Riemannian and pseudoriemannian geometry: by  \cite{AlekseevskiKimelFel}, homogeneous Riemannian Ricci-flat manifolds are flat. Parak\"ahler Einstein manifolds with non-zero scalar curvature that are homogeneous under a semisimple Lie group are classified in \cite{AlekseevskyMedoriTomassini:Homogeneous2009}.

Outside  the parak\"ahler setting, one significant intrinsic torsion class consists of nearly parak\"ahler structures, characterized by the fact that $\nabla K$ is skew-symmetric in the first two indices. In our language, this means that the $\GL(n,\R)$ intrinsic torsion is contained in $W_1\oplus W_5$ at each point; we say that its intrinsic torsion \emph{class} is $\mathcal{W}_1+\mathcal{W}_5$. Examples of Einstein and Ricci-flat nearly parak\"ahler metrics are constructed in  \cite{IvanovZamkovoy:parahermitian}, \cite{CortesSchafer} and \cite{Schafer:conical}. In fact, it follows easily from our formula  (see Corollary~\ref{cor:nk}) that nearly parak\"ahler manifolds of dimension six are automatically Einstein, as originally proved in \cite{IvanovZamkovoy:parahermitian}.

As a more restrictive condition, we study the intrinsic torsion class $\mathcal{W}_1$; we show that this is Ricci-flat in all dimensions. Considering left-invariant structures on nilpotent Lie groups,  we obtain several explicit examples in dimension eight (Theorem~\ref{thm:nkricciflat}). These structures are automatically nearly parak\"ahler and Ricci-flat; they are also non-flat, and the underlying manifold is compact. We note that the previously known examples of Ricci-flat nearly parak\"ahler manifolds were either non-compact (\cite{Schafer:conical}) or flat (\cite{CortesSchafer}).

The intrinsic torsion class $\mathcal{W}_2$ is also Ricci-flat. This leads to a counterexample of the  paracomplex version of the Goldberg conjecture as stated in  \cite{Matsushita}, asserting that a compact, Einstein almost parak\"ahler manifold is necessarily parak\"ahler (Proposition~\ref{prop:goldberg}). Notice that the K\"ahler version of the conjecture is known to hold for non-negative scalar curvature (\cite{Sekigawa}); our example is Ricci-flat, showing that the paracomplex situation is different.

The intrinsic torsion class $\mathcal{W}_3$ is not Ricci-flat, but we are able to construct a compact Ricci-flat example on a nilmanifold. However, the class $\mathcal{W}_4$ is different: a nilpotent Lie group with an invariant structure with intrinsic torsion in $W_4+W_8$ is necessarily parak\"ahler (Proposition~\ref{prop:nilpotentw4w8}).

Observing that changing the sign of $K$ has the effect of swapping $W_i$ with $W_{i+4}$, this concludes the analysis of ``pure'' intrinsic torsion classes. By taking products, it follows that all intrinsic torsion classes that do not contain $\mathcal{W}_4$ or $\mathcal{W}_8$ can be realized as the intrinsic torsion class of a nilmanifold with a non-flat, Ricci-flat metric (Proposition~\ref{prop:ricciflatproducts}).

\section{The structure group $\GL(n,\R)$}
An almost paracomplex structure on a manifold of dimension $2n$ is a decomposition of the tangent space in two subbundles of rank $n$. The tangent space is then modeled on a direct sum
\begin{equation}
 \label{eqn:linearmodel}
T=V\oplus H,
\end{equation}
where $V$ and $H$ are real vector spaces of dimension $n$; explicitly, we shall fix a basis $e_1,\dotsc, e_{2n}$ of $T$ with
\[V=\Span{e_1,\dotsc, e_n},\quad H=\Span{e_{n+1},\dotsc, e_{2n}},\]
and denote by $e^1,\dotsc, e^{2n}$ the dual basis of $T^*$. In these terms, we can think of an almost paracomplex structure as a $\GL(V)\times \GL(H)$-structure. In analogy with complex geometry, one considers an endomorphism of $T$ with $K^2=\id$, namely
\[K =\id_V - \id_H = e^i\otimes e_i - e^{n+j}\otimes e_{n+j};\]
here and in the sequel, summation over repeated indices is implied; we adopt the convention that lower case indices range from $1$ to $n$, and upper case indices range from $1$ to $2n$. 

It is clear that a manifold of dimension $2n$ admits an almost paracomplex structure if and only if it admits a distribution of rank $n$; for example, $S^{2n}$ does not have an almost paracomplex structure \cite[Theorem~27.18]{Steenrod}. Thus, the long-standing problem of whether the six-dimensional sphere admits an integrable complex structure has a trivial answer in the paracomplex setting.

Like in almost complex geometry, differential forms on an almost paracomplex structure can be decomposed according to type via
\[\Lambda^k T^*=\bigoplus_{p+q=k} \Lambda^{p,q}, \quad \Lambda^{p,q} = \Lambda^pV^*\otimes \Lambda^qH^*;\]
in the literature, one also finds the notations $TM=T_+M\oplus T_-M$ and $\Lambda^{p,q}_\pm$.

An almost paracomplex structure is said to be \emph{paracomplex} or \emph{integrable} if the two rank $n$ distributions are integrable. By the Frobenius Theorem, this is equivalent to requiring that the exterior derivative have the form
\[d\colon \Lambda^{p,q}\to \Lambda^{p+1,q}+\Lambda^{p,q+1},\]
or to the vanishing of the Nijenhuis tensor
\[N(X,Y)=[X,Y]+[KX,KY]-K[KX,Y]-K[X,KY].\]
In the language of $G$-structures, integrability can be expressed in terms of \emph{intrinsic torsion}. Recall that the intrinsic torsion of a $G$-structure takes values in the cokernel of the linear map $\partial_G$, defined as the restriction to $T^*\otimes\lie{g}$ of the alternating map
\[\partial\colon T^*\otimes \gl(T)\to \Lambda^2T^*\otimes T, \quad e^i\otimes (e^j\otimes e_k) \to e^{ij}\otimes e_k;\]
more precisely, the intrinsic torsion is obtained by projecting on this space the torsion of any connection on the $G$-structure. 

If $\partial_G$ has a left inverse $s$, any $G$-structure has a unique \emph{minimal} connection, namely one with torsion in $\ker s$. In the present case, the alternating map is not injective, but we still have the following:
\begin{proposition}
\label{prop:paracomplexit}
Every $\GL(V)\times \GL(H)$-structure admits a connection with torsion taking values in 
\begin{equation}
 \label{eqn:paracomplexit}
\Lambda^2V^*\otimes H+\Lambda^2H^*\otimes V;
\end{equation}
the torsion $\Theta$ of any such connection  is related to the Nijenhuis tensor via
\[\Theta(X,Y)=-\frac14 N(X,Y).\]
\end{proposition}
\begin{proof}
The representation $\Lambda^2T^*\otimes T$ splits as
\[\Lambda^{2,0}\otimes V+\Lambda^{2,0}\otimes H+\Lambda^{1,1}\otimes V+\Lambda^{1,1}\otimes H+\Lambda^{0,2}\otimes V+\Lambda^{0,2}\otimes H,\]
and
\[T^*\otimes (\gl(V)+\gl(H))= (S^2V^*\otimes V+S^2H^*\otimes H) + (\Lambda^{2,0}\otimes V + \Lambda^{0,2}\otimes H+\Lambda^{1,1}\otimes T),\]
where the first component represents the kernel of the alternating map. It follows that the restriction of $\partial$ to the second component  has a left inverse $s$ with kernel \eqref{eqn:paracomplexit}. Therefore, if $\overline\omega$ is any connection on the $\GL(V)\times \GL(H)$-structure with torsion $\overline\Theta$, the connection $\omega=\overline\omega-s(\overline\Theta)$
has torsion $\overline\Theta - \partial (s(\overline\Theta))\in \ker s.$
Applying the definition
$\Theta(X,Y)=\nabla_X Y-\nabla_Y X-[X,Y],$
and using the fact that $V$ and $H$ are preserved by the connection $\omega$,  we obtain the required relation.
\end{proof}
\begin{remark}
In terms of an adapted coframe $e^1,\dotsc,e^{2n}$, the torsion of a connection such as in Proposition~\ref{prop:paracomplexit} may be written as
\[\Theta = (de^i)^{0,2}\otimes e_i + (de^{n+i})^{2,0}\otimes e_{n+i}.\]
\end{remark}

\smallskip
Given an almost paracomplex structure,  a pseudoriemannian metric $g$ is called \emph{almost parahermitian} if
\[g(KX,KY)=-g(X,Y).\]
Such a metric is necessarily of neutral signature $(n,n)$; relative to  \eqref{eqn:linearmodel}, it determines an isomorphism $V\cong H^*$. Accordingly, the structure group is reduced to $\GL(n,\R)$, and the tangent space is modeled on the representation
\[T=V\oplus V^*.\]
Alternatively, we may think of an almost parahermitian structure as determined by a non-degenerate two-form $F$ of type $(1,1)$; form and metric are related via
\[F(X,Y)=g(KX,Y).\]
In this context, the basis of $T$ can always be chosen so that 
\[g=e^i\odot e^{n+i}, \quad F=e^{1,n+1}+\dots + e^{n,2n},\]
where $e^i\odot e^{n+i}$ stands for $e^i\otimes e^{n+i}+e^{n+i}\otimes e^i$ and $e^{i,n+i}$ for $e^i\wedge e^{n+i}$;
using the metric $g$, we will use the identification
\[V^*\ni e^i\mapsto e_{n+i}\in H.\]

Due to the existence of the volume form $F^n$, almost parahermitian manifolds are orientable.

Whilst a paracomplex manifold is always locally a product $M\times N$, one should not think of parahermitian geometry as a fancy way to describe Cartesian products. At the topological level, this can be seen from the following:
\begin{proposition}
\label{prop:noproduct}
Let $M$ and $N$ be manifolds of dimension $n$, and assume that $TM$ is not trivial. Then the product paracomplex structure on $M\times N$ does not admit a compatible parahermitian structure.
\end{proposition}
\begin{proof}
Let 
\[\pi_1\colon M\times N\to M, \quad \pi_2\colon M\times N\to N\] denote the projections. The $K$-eigenspaces for the product paracomplex structure on  $M\times N$ are $\pi_1^*TM$ and $\pi_2^*TN$. If a compatible parahermitian structure exists, the vector bundles
$\pi_1^*TM \cong (\pi_2^*TN)^*$
are isomorphic. Therefore, their restrictions to a submanifold $M\times\{y\}$ are also isomorphic. However, the restriction of $\pi_1^*TM$ is equivalent to $TM$, and the restriction of $(\pi_2^*TN)^*$ is trivial, which is absurd.
\end{proof}

At a point, we can think of the structure group $\GL(n,\R)$ as the stabilizer of $F$ in $\SO(n,n)$. At the Lie algebra level, this amounts to setting $B$ and $C$ to zero   in 
\[\so(n,n)=\biggl\{\begin{pmatrix} A & B \\ C & -\tran{A}\end{pmatrix}\mid  B=-\tran{B}, C=-\tran{C}\biggr\}.\]
Having a metric at our disposal, we can write orthogonal decompositions such as 
\[\so(n,n)=\gl(n,\R)\oplus \gl(n,\R)^\perp=\Sl(n,\R)\oplus \Sl(n,\R)^\perp,\]
where 
\[
\Sl(n,\R)=\{Tr(A)=0, B=0=C\},\quad  \Sl(n,\R)^\perp=\{A=\lambda I\}.\]
It will be convenient to fix the isomorphism
\begin{equation}
 \label{eqn:lambdaisso}
\Lambda^2T^*\ni \alpha\mapsto M_\alpha\in \so(n,n),  \quad \langle M_\alpha(v),w\rangle =\alpha(v,w).
\end{equation}
Explicitly,
\begin{gather*}
e^{ij}\mapsto e^i\otimes e_{n+j}-e^j\otimes e_{n+i}, \qquad e^{n+i,n+j}\mapsto e^{n+i}\otimes e_j-e^{n+j}\otimes e_i,\\
e^{i,n+j}\mapsto e^i\otimes e_j-e^{n+j}\otimes e_{n+i}.
\end{gather*}
\begin{lemma}
\label{lemma:glglscalar}
Through the identification \eqref{eqn:lambdaisso}, the Lie bracket on $\so(n,n)$ satisfies
\begin{equation*} 
[a_{ij}e^{ij},b_{kl}e^{n+k,n+l}] =\frac2n a_{ij}(b_{ij}-b_{ji}) e^{k,n+k} \mod \Sl(n,\R) + \gl(n,\R)^\perp.
\end{equation*}
\end{lemma}
\begin{proof}
Follows from 
\[[e^{ij},e^{n+k,n+l}] =
-\delta_{il}e^{j,n+k} +\delta_{jl}e^{i,n+k} +\delta_{ik}e^{j,n+l}-\delta_{jk}e^{i,n+l}.\qedhere\]
\end{proof}

An almost parahermitian structure is called \emph{parak\"ahler} if one (hence both) of $K$ and $F$ is parallel under the Levi-Civita connection. More generally, $\nabla F$ can be identified with  the intrinsic torsion of a $\GL(n,\R)$-structure; the latter is known to decompose into eight components \cite{GadeaMasque}.

All finite-dimensional irreducible representations of $\GL(n,\R)$ appear inside some
\begin{equation}
 \label{eqn:Vrs}
V^{\otimes^r}\otimes (V^*)^{\otimes ^s},
\end{equation}
(see e.g. \cite{FultonHarris}). Relative to the Cartan subalgebra of diagonal matrices, let $L_i$ denote the weight that maps $\operatorname{diag}(a_1,\dotsc, a_n)$ to $a_i$. For any $k+h\leq n$ and integers
$\lambda_1\geq \dotsc\geq \lambda_k$,  $\lambda_n\geq \dotsc \geq \lambda_{n-h}$,
denote by $V_{\lambda_1,\dotsc, \lambda_k}^{\lambda_{n},\dotsc, \lambda_{n-h}}$ the representation with highest weight
\[\lambda_1 L_1 + \dotsc + \lambda_k L_k - (\lambda_{n-h}L_{n-h} + \dotsc + \lambda_n L_n).\]
These representations are also irreducible under $\SL(n,\R)$, but notice that under $\SL(n,\R)$ the numbers  $\lambda_1,\dotsc, \lambda_n$ are not determined uniquely by the representation.

With this notation, $V=V_1$ and $V^*=V^1$. More generally, 
\[\big(V_{\lambda_1,\dotsc, \lambda_k}^{\lambda_{n},\dotsc, \lambda_{n-h}}\big)^* =V^{\lambda_1,\dotsc, \lambda_k}_{\lambda_{n},\dotsc, \lambda_{n-h}};\]
in addition we have
\[\lie{sl}(n,\R)=V_1^1, \quad \Lambda^kV = V_{ \underbrace{\scriptstyle 1,\dots,1}_{\scriptstyle{k}}}, \quad S^kV=V_k.\]
We shall say that a representation has \emph{type} $(k,h)$ if it is the sum of irreducible representations of the form $V_{\lambda_1,\dotsc, \lambda_k}^{\lambda_{n},\dotsc, \lambda_{n-h}}$. In terms of Young diagrams, this says that the rows from $k+1$ to $n-h-1$ have the same length. It is clear that  $V^{\otimes^r}$ has type $(r,0)$; dually, $(V^*)^{\otimes^s}$  has type $(0,s)$. It now follows easily from the Littlewood-Richardson rule that \eqref{eqn:Vrs} has type $(r,s)$.

We can think of a representation of type $(r,s)$ as a representation of $\GL(n,\R)$ for any choice of $n>r+s$. The decomposition into irreducible components is then independent of $n$. For instance, for $s\geq r$ we have
\begin{gather*}
\Lambda^rV\otimes \Lambda^s V^* = V_{ \underbrace{{\scriptstyle{1,\dots,1}}}_{\scriptstyle{r}}}^{\overbrace{\scriptstyle1,\dotsc, 1}^{\scriptstyle{s}}} +  V_{ \underbrace{\scriptstyle1,\dots,1}_{\scriptstyle{r-1}}}^{\overbrace{\scriptstyle1,\dotsc, 1}^{\scriptstyle{s-1}}}+\dots + V^{\overbrace{\scriptstyle1,\dotsc, 1}^{\scriptstyle{s-r}}};
\end{gather*}
this also holds for $r+s=n$. Equivalently, we can write
\[\Lambda^{r,s}=\Lambda^{r,s}_0 + \{F\wedge\sigma,\sigma\in \Lambda^{r-1,s-1}\} \cong \Lambda^{r,s}_0 +\dots + \Lambda^{1,s-r+1}_0+ \Lambda^{0,s-r}.\]
In terms of an appropriate map $\Lambda\colon \Lambda^{r,s}\to\Lambda^{r-1,s-1}$, we have
\[\sigma=[\sigma]_0 + F\wedge \Lambda(\sigma), \quad \sigma\in\Lambda^{r,s}.\]
Explicitly,
\begin{equation}
 \label{eqn:LambdaOn21}
\Lambda(\gamma)=-\frac1{n-1} e_i\hook e_{n+i}\hook\gamma, \quad \gamma\in\Lambda^{2,1}+\Lambda^{1,2}.
\end{equation}

As a consequence of the Littlewood-Richardson rule, one obtains:
\begin{proposition}[\cite{GadeaMasque}]
The intrinsic torsion of a $\GL(n,\R)$-structures lies in 
\[T^*\otimes \lie{gl}(n,\R)^\perp=\coker \partial_{\GL(n,\R)}\cong W_1+\dots + W_8\]
where
\begin{gather*}
W_1=\Lambda^3V^*,\quad W_2=V^{2,1},\quad W_3=V_{1,1}^1 ,\quad W_4=V,\\
W_5=\Lambda^3V,\quad W_6=V_{2,1},\quad W_7=V^{1,1}_1 ,\quad W_8=V^*.
\end{gather*}
\end{proposition}

To compare with Proposition~\ref{prop:paracomplexit}, observe that the $\GL(V)\times\GL(H)$-intrinsic torsion corresponds to $W_1+W_2+W_5+W_6$; however, this space only splits in two irreducible components under the enlarged structure group.

\section{$\SL(n,\R)$-structures}
In this section we turn to the structure group $\SL(n,\R)$. The invariant elements of $\Lambda^*T$  under the action $\SL(n,\R)$ are generated by
\[F=e^{1,n+1}+\dots + e^{n,2n}, \quad \alpha=e^{1,\dots, n} , \quad \beta=e^{n+1,\dots, 2n};\]
conversely, $\SL(n,\R)$ is the largest group that fixes this subalgebra.

Given a $\GL(n,\R)$-structure, a reduction to $\SL(n,\R)$ is determined by a global, nowhere-vanishing form of type $(n,0)$. Clearly, a $\GL(n,\R)$-structure admits a reduction to $\SL(n,\R)$ if and only if it admits a reduction to $\GL_+(n,\R)$, i.e. the two rank $n$ distributions are orientable. 

\begin{example}
Consider the standard parahermitian structure on $\R^4=\R^2\times \R^2$. This structure is preserved by the group $\Gamma$ generated by the diffeomorphism
\[(x,y,z,t)\mapsto (x+1,-y,z,-t).\]
Hence the quotient $\R^4/\Gamma$ has an induced parahermitian structure. In this case, the rank two subbundles are not orientable, and there is no global reduction to $\SL(n,\R)$.
\end{example}

\begin{remark}
Given an almost parahermitian structure on $M$, namely a $\GL(n,\R)$-structure $P$, it is always possible to find a $2:1$ cover of $M$ which admits a reduction to $\SL(n,\R)$. Indeed, the quotient $P/\GL_+(n,\R)$ is a $2:1$ cover of $M$ which admits a tautological $\GL_+(n,\R)$-structure.
\end{remark}

\smallskip
The intrinsic torsion of a $\SL(n,\R)$-structure takes values in the cokernel of the alternating map
\[ \partial_{\SL(n,\R)}\colon T^*\otimes \lie{sl}(n,\R)\to \Lambda^2T^*\otimes T;\]
since $\partial_{O(n,n)}$ is an isomorphism, we can identify this space with $T^*\otimes \Sl(n,\R)^\perp$.
\begin{proposition}
The intrinsic torsion of an $\SL(n,\R)$-structure lies in
\[T^*\otimes \Sl(n,\R)^\perp = T^*\otimes \gl(n,\R)^\perp \oplus W^{1,0}\oplus W^{0,1},\]
where \[W^{1,0}=\partial (V^*\otimes \R)\cong V^*, \quad W^{0,1}=\partial (V\otimes \R)\cong V.\]
\end{proposition}

The Levi-Civita connection can always be written in the form
\begin{equation}
\label{eqn:omegataulambda} 
\omega^{LC}=\omega+\tau+\tilde\lambda,
\end{equation}
with $\tau \in T^*\otimes \lie{gl}(n,\R)^\perp$, $\omega$ is an almost-parahermitian connection and
\[\tilde\lambda=\lambda\otimes \begin{pmatrix} I & 0 \\ 0 & -I\end{pmatrix},\]
where $\lambda$ is a one-form. For future reference, we note that given a form $\sigma$,
\begin{equation}
\label{eqn:bracketlambda}
\partial(\tilde\lambda)\hook \sigma=\sum_{p,q} (p-q)\lambda\wedge \sigma^{p,q}.
\end{equation}

Thus, $\omega$ is the connection obtained from the Levi-Civita connection by projection on $\Sl(n,\R)$; we shall refer to it as the \emph{minimal} connection, and denote by $\nabla, D$ the corresponding covariant derivative and exterior covariant derivative. By construction, the torsion of $\omega$ is $\Theta=-\partial(\tau+\tilde\lambda)$. The component  $\tau$ can be decomposed as the sum of
\[\tau=\tau_1+\dots + \tau_8,\]
with each $\tau_i$ corresponding to a section of the bundle associated to $W_i$. Relative to the action of $\R^*\subset\GL(n,\R)$, we can decompose $\tau$ into four components with weights $-3,1,3,-1$, namely
\begin{equation}
 \label{eqn:tauindices}
\begin{gathered}
\tau_1+\tau_2= a_{ijk}e^i\otimes e^{jk}\in V^*\otimes \Lambda^2V^*,\quad
 \tau_3+\tau_4= c_{ijk}e^i\otimes e_{jk}\in V^*\otimes \Lambda^2V,\\
 \tau_5+\tau_6= b_{ijk}e^{n+i}\otimes e_{jk}\in V\otimes \Lambda^2V,\quad
 \tau_7+\tau_8= d_{ijk}e^{n+i}\otimes e^{jk}\in V\otimes \Lambda^2V^*.
\end{gathered}
\end{equation}
Here, summation over all $i,j,k$ is implied, and we assume that $a_{ijk}=-a_{ikj}$.

The components $W_4$, $W_8$, $W^{1,0}$ and $W^{0,1}$ can be encoded in three one-forms
\begin{equation}
 \label{eqn:pag7alto}
f_4 = a_ie^{n+i},  \quad  f_8=b_ie^{i}, \quad \lambda=\lambda_I e^I 
\end{equation}
characterized by
\begin{equation}
 \label{eqn:tau4tau8lambda}
\tau_4 = a_i(e^k\otimes e_{ki}),\quad \tau_8=b_i(e^{n+k}\otimes e^{ki}), \quad \tilde\lambda = \lambda_I e^I\otimes (e^k\otimes e_k - e^{n+k}\otimes e_{n+k}).
\end{equation}

A useful symmetry arises as follows. Given an $\SL(n,\R)$-structure $P$, one can consider the  $\SL(n,\R)$-structure $P\sigma$, where  $\sigma=\left(\begin{smallmatrix}  0 & I \\ I & 0 \end{smallmatrix}\right);$
this amounts to  interchanging $V$ and $H$. An adapted coframe $e^1,\dots, e^{2n}$ for $P$ determines an adapted coframe
\[e_\sigma^1,\dots, e_\sigma^{2n} = e^{n+1},\dots, e^{2n},e^1,\dots, e^n,\]
relative to which the intrinsic torsion has the form
\begin{gather*}
(a_\sigma)_{ijk} = b_{ijk},\quad(b_\sigma)_{ijk} = a_{ijk},\quad (c_\sigma)_{ijk} = d_{ijk},\quad (d_\sigma)_{ijk} = c_{ijk},\\
(\lambda_\sigma)_i = -\lambda_{n+i}, \quad (\lambda_\sigma)_{n+i}=-\lambda_i; 
\end{gather*}
the minus sign originates from the action of $\sigma$ on $\left(\begin{smallmatrix} I & 0 \\ 0 & - I\end{smallmatrix}\right)$.

There are constraints on $\tau$ coming from the first Bianchi identity. Indeed, recall that given a tensorial $k$-form $\eta$, one has
\begin{equation}
 \label{eqn:exteriorcovariant}
D\eta = \mathfrak{a}(\nabla\eta)+\Theta\hook\eta,
\end{equation}
where $\Theta$ is the torsion and 
$\mathfrak{a}(\nabla\eta)=\langle \nabla\eta, \frac1{k!}\theta\wedge \dots \wedge \theta \rangle .$
In particular,
\[D\Theta = D(-\partial(\tau+\tilde\lambda))=\mathfrak{a}(-\nabla\partial(\tau+\tilde\lambda))+\partial(\tau+\tilde\lambda)\hook \partial(\tau+\tilde\lambda)\]
must satisfy
\[D\Theta=\Omega\wedge\theta\in \Lambda^2T^*\otimes \Sl(n,\R) \subset\Lambda^3T^*\otimes T.\]
Define the equivariant maps
\begin{align*}
p&\colon\Lambda^3T^*\otimes T\to \Lambda^{1,1}_0, & \eta\otimes v&\mapsto [v\hook \eta +(n-1) \Lambda\eta\wedge v\hook F]_{\Lambda^{1,1}_0};\\
q&\colon\Lambda^3T^*\otimes T\to \R, & \eta\otimes v&\mapsto \Lambda( v\hook \eta).
\end{align*}
\begin{proposition}
The intrinsic torsion of an $\SL(n,\R)$-structure satisfies
\label{prop:constraint}
\begin{multline*}
 p\bigl(\mathfrak{a}(-\nabla\partial(\tau_3+\tau_4+\tau_7+\tau_8))+ \partial(\tau_1)\hook \partial(\tau_6)
+ \partial(\tau_2)\hook \partial(\tau_5)
+ \partial(\tau_6)\hook \partial(\tau_1)\\
+ \partial(\tau_5)\hook \partial(\tau_2)
+\partial(\tau_3+\tau_4+\tau_7+\tau_8)\hook \partial(\tilde\lambda )\bigr)
=0
\end{multline*}
\begin{equation*}
q(\mathfrak{a}(-\nabla\partial(\tau_4+\tau_8))+\partial(\tau_4+\tau_8)\hook \partial(\tilde\lambda )\bigr)=0.
\end{equation*}
\end{proposition}
\begin{proof}
As a first step, we prove that $p$ and $q$ kill $\Omega\wedge\theta$. Decomposing $\Lambda^3T^*$ under $\GL(n,\R)$, one readily sees that $p$ and $q$ are only non-trivial on $\Lambda^{2,1}\otimes V+\Lambda^{1,2}\otimes V^*$. Writing $\Omega=\Omega^{2,0}+\Omega^{1,1}+\Omega^{0,2}$, it follows that
\[p(\Omega\wedge\theta)=p(\Omega^{1,1}\wedge\theta).\]
By linearity, we can assume that $\Omega^{1,1}$ has the form $e^{i,n+j}\otimes (e^k\otimes e_{h}-e^{n+h}\otimes e_{n+k})$; using \eqref{eqn:LambdaOn21}, we obtain
\begin{multline*}
p(\Omega^{1,1}\wedge\theta)=p( e^{i,n+j,k}\otimes e_h - e^{i,n+j,n+h}\otimes e_{n+k})\\
=[\delta_{ih} e^{n+j,k} +e^{i,n+h}\delta_{kj}+e_j\hook e^{i,k,n+h} +e^{n+i}\hook e^{n+j,n+h,k}]_{\Lambda^{1,1}_0}=0.
\end{multline*}
The Bianchi identity now implies
\begin{equation*}
0=p(\Omega\wedge\theta)=p(\mathfrak{a}(-\nabla\partial(\tau+\tilde\lambda))+\partial(\tau+\tilde\lambda)\hook \partial(\tau+\tilde\lambda)),
\end{equation*}
and the same holds for $q$.

From $\lambda\wedge\lambda=0$, we obtain $\partial(\tilde\lambda)\hook\partial(\tilde\lambda)=0$; the component
$\partial(\tilde\lambda)\hook\partial(\tau)$
gives no contribution because of Schur's lemma and \eqref{eqn:bracketlambda}. Similarly,  $\nabla(\partial(\tau_1+\tau_2+\tau_5+\tau_6))$ is in the kernel of $p$ and $q$ because it has no component in $\Lambda^{2,1}\otimes V+\Lambda^{1,2}\otimes V^*$.

Observe that
\[\mathfrak{a}(\nabla\partial\tilde\lambda) = d\lambda\wedge e^k\otimes e_k - d\lambda\wedge e^{n+k}\otimes e_{n+k},\]
hence,  using \eqref{eqn:LambdaOn21} again, 
\begin{multline*}
p(\mathfrak{a}(\nabla\partial\tilde\lambda)= \bigl[e_k\hook (d\lambda\wedge e^k) - e_{n+k}\hook (d\lambda\wedge e^{n+k}) \\
+(n-1)\Lambda(d\lambda \wedge e^k)\wedge e^{n+k}+(n-1) \Lambda(d\lambda\wedge e^{n+k})\wedge e^k)\bigr]_{\Lambda^{1,1}_0}\\
=\bigl[e_k\hook (d\lambda\wedge e^k) - e_{n+k}\hook (d\lambda\wedge e^{n+k})
-e_k\hook (d\lambda\wedge e^k)+e_{n+k}\hook (d\lambda\wedge e^{n+k})\bigr]_{\Lambda^{1,1}_0}=0;
\end{multline*}
a similar calculation shows that $q(\mathfrak{a}(\nabla\partial\tilde\lambda)=0$.

Writing 
\begin{multline*}
\partial(\tau_3+\tau_4)\hook \partial(\tau_7+\tau_8)= 4c_{ijk}e^{i,n+j}\otimes e_{k} \hook (d_{hlm}e^{n+h,l}\otimes e_{n+m})\\
=-4c_{ijk}d_{hkm}e^{i,n+j, n+h}\otimes e_{n+m},
\end{multline*}
and symmetrically
$\partial(\tau_7+\tau_8)\hook \partial(\tau_3+\tau_4)= -4d_{ijk} c_{hkm}e^{n+i,j,h}\otimes e_m,$
we find that
\begin{multline*}
p(\partial(\tau_7+\tau_8)\hook \partial(\tau_3+\tau_4) + \partial(\tau_7+\tau_8)\hook \partial(\tau_3+\tau_4))\\
= \bigl[4c_{imk}d_{hkm}e^{i, n+h} -4c_{ijk}d_{hkh}e^{i,n+j} + 4c_{iik}d_{hkm}e^{ n+h,m}-   4c_{ijk}d_{ikm}e^{n+j,m}\\
+ 4d_{imk}c_{hkm}e^{n+i, h} -4d_{ijk}c_{hkh}e^{n+i,j} + 4d_{iik}c_{hkm}e^{h,n+m}-   4d_{ijk}c_{ikm}e^{j,n+m}\bigr]_{\Lambda^{1,1}_0}
\end{multline*}
is zero, and the same for $q$. 

Explicit computations shows that $p$ and $q$ annihilate $\partial(\tau_2)\hook \partial(\tau_6)\allowbreak+\allowbreak\partial(\tau_6)\hook \partial(\tau_2)$ and
$\partial(\tau_1)\hook \partial(\tau_5)+\partial(\tau_5)\hook \partial(\tau_1)$.

Finally, observe that  $\nabla\partial\tau_i$ lies in a module isomorphic to $T^*\otimes W^i$ which only contains a component isomorphic to $V^1_1$ for $i=3,4,7,8$, and to $\R$ for $i=4,8$. Similarly, $\partial(\tau_i)\hook\partial(\tau_j)$ lies in a module isomorphic to $W^i\otimes W^j$; the equivariance of $p$ and $q$ gives the statement.
\end{proof}

\begin{remark}
The component $\tau$ of the intrinsic torsion depends only on the $\GL(n,\R)$-struc\-ture; the component $\lambda$ only depends on the $\SL(n,\R)\times \SL(n,\R)$-structure. The components $\tau_1,\tau_2,\tau_5,\tau_6$ depend on the paracomplex structure (see Proposition~\ref{prop:paracomplexit}); more precisely, they are determined by
\[(de^i)^{0,2}, \quad (de^{n+i})^{2,0}.\]
\end{remark}

\section{Ricci curvature}
In this section we give a formula for the Ricci tensor of an $\SL(n,\R)$-structure, expressed in terms of its intrinsic torsion. Even though the structure group is $\SL(n,\R)$, all relevant representations have a natural action of $\GL(n,\R)$, and the maps we consider in this section are $\GL(n,\R)$-invariant; accordingly, we will regard two representations as isomorphic if they are under this larger group.

The basic idea  is that, relative to the decomposition
\[\Lambda^2T^*\otimes \so(n,n)= \Lambda^2T^*\otimes (\Sl(n,\R)\oplus \gl(n,\R)^\perp\oplus \R),\]
the relevant part of the curvature of the Levi-Civita connection is determined by the last two components, which only depend on the intrinsic torsion. Indeed, writing the Levi-Civita connection form as in \eqref{eqn:omegataulambda}, its curvature decomposes as
\[\Omega^{LC}=(\Omega+\frac12[[\tau,\tau]]_{\Sl(n,\R)}) + (D\tau + [\tilde\lambda,\tau] )+ (d\lambda-\frac1n F(\tau,\tau)) \otimes \begin{pmatrix} I  & 0 \\ 0 & -I\end{pmatrix},\]
where we have used Lemma \ref{lemma:glglscalar}. Here $D$ denotes the exterior covariant derivative of the minimal connection, and $F(\tau,\tau)$ denotes a $2$-form obtained by contracting $\tau$ with itself using $F$. More precisely, let $\langle\cdot,\cdot\rangle$ denote the natural pairing between $V$ and $V^*$, and consider the skew bilinear form on $\Lambda^2V+\Lambda^2V^*$ such that
\[\tilde F(\gamma,\sigma)=\langle \gamma,\sigma
\rangle=-\tilde F  (\sigma,\gamma), \gamma\in\Lambda^2V, \sigma\in\Lambda^2V^*,\]
and zero otherwise,  and define \[F(\eta\otimes \gamma, \eta'\otimes \gamma')=\tilde F(\gamma,\gamma') \eta\wedge \eta'.\]
Thus, for fixed $i,h,j\neq k$,
\[F(e^i\otimes e_{jk}, e^{n+h}\otimes e_{n+j,n+k}) = e^i\wedge e^{n+h}.\]
We shall also consider the $2$-form $\overline{F}(\tau,\tau)$ defined by
\[\overline{F}(e^I\otimes e^J\otimes e^K,e^H\otimes e^L \otimes e^M)=-F(e^I,e^H) F(e^J,e^L) e^K\wedge e^M.\]

We shall decompose the Ricci tensor of the Levi-Civita connection as
\[\ric=\ric'+\ric'',\quad \ric'\in V^*\otimes V, \quad \ric''\in S^2V\oplus S^2V^*.\]
Here, $V^*\otimes V$ represents a subspace of $S^2T^*$, i.e. $e^{i}\otimes e^{n+j}$ stands for $e^{i}\odot e^{n+j}$. We will also identify this space with $\Lambda^{1,1}$ through
\begin{equation}
\label{eqn:identify11}
V^*\otimes V\cong \Lambda^{1,1}, \quad e^i\otimes e^{n+j}\mapsto e^{i,n+j}.
\end{equation}

As a first approximation, the $V^*\otimes V$ part of the Ricci can be described as follows.
\begin{lemma}
The Ricci tensor of an $\SL(n,\R)$-structure satisfies
\label{lemma:ricprime}
\[\ric' = 2\ric' (D\tau+[\tilde\lambda, \tau])+n(d\lambda)^{1,1}-F( \tau,\tau)^{1,1}.\]
\end{lemma}
\begin{proof}
The Riemann tensor takes values in the kernel $\mathcal{R}$ of the skewing map
\begin{equation*}
S^2(\Lambda^2T^*)\to\Lambda^4T^*.
\end{equation*}
As a $\GL(n,\R)$-module, $S^2(\Lambda^2T^*)$ decomposes as
\[  S^2(\Lambda^2T^*)=3\R +3V_1^1+2V_{1,1}^{1,1}  + V_2^2+U+U^*,\]
where $U=2\Lambda^2V+\Lambda^4V+V_{2,2}+V_{1,1,1}^1  + V_{2,1}^1 + S^2V$. Since the Ricci contraction is equivariant, and $\ric'$ takes values in $V^*\otimes V=V_1^1\oplus\R$, we only need to consider the components of $\mathcal{R}$ isomorphic to $V_1^1$ and $\R$.

The three components of $S^2(\Lambda^2T^*)$ isomorphic to $V^1_1$ contain the highest weight vectors 
\[
 v_1=
 e^{n,k+n}\odot  e^{k,n+1},\quad  v_2=e^{n,n+1}\odot e^{k,n+k}, \quad  v_3=e^{n+1,k+n}\odot e^{kn},
\]
where $nv_1-2v_2$ lies in $S^2(\Sl(n,\R))$, $v_2$ in $\Sl(n,\R)^\perp\otimes \Sl(n,\R)$ and $v_3$ in $ S^2(\Sl(n,\R)^\perp)$.
However $\mathcal{R}$ only contains two copies of $V_1^1$, generated by $v_1+v_2, v_1+v_3.$

Similarly, the components isomorphic to $\R$ in $S^2(\Lambda^2T^*)$ are generated by 
\begin{gather*}
w_1= e^{n+i,n+j}\otimes e^{ij}+ e^{ij}\otimes e^{n+i,n+j}\in S^2(\gl(n,\R)^\perp),\\
w_2= e^{i,n+i}\otimes e^{j,n+j}\in S^2(\R)
\end{gather*}
and
\[w_3-\frac1nw_2\in S^2(\Sl(n,\R)), \quad w_3=\sum_{i,j} e^{j,n+i}\otimes e^{i,n+j}.\]
The vectors $w_1+2w_2$, $w_2+w_3$ generate the two copies of $\R$ in $\mathcal{R}$. 

By equivariance, and neglecting components not isomorphic to $V_1^1$ and $\R$, which do not contribute to $\ric'$, we may assume that the Riemann tensor has the form
\[R=a(v_1+v_2)+b(v_1+v_3)+h(w_1+2w_2) + k(w_2+w_3).\]
The Ricci contraction of the fixed generators is given by
\begin{gather*}
 \ric(w_1)=-2(n-1)g, \quad \ric(w_2)=g, \quad \ric(w_3)=ng,\\
\ric(v_1)=ne^n\odot e^{n+1}, \quad \ric(v_2)=2e^n\odot e^{n+1},\quad \ric(v_3)=(n-2)e^n\odot e^{n+1};
\end{gather*}
with our choice of $R$, the Ricci tensor is
\begin{equation}
 \label{eqn:ricprime}
\ric' = (a(n+2)+b(2n-2))e^n\odot e^{n+1}+(2(2-n)h+(1+n)k)g.
\end{equation}
Consider the projections
\begin{gather*}
\pi_{\gl^\perp}\colon \mathcal{R}\to \Lambda^2T^*\otimes \gl(n,\R)^\perp,\qquad
\pi_{\Sl}\colon \mathcal{R}\to \Lambda^2T^*\otimes \Sl(n,\R),\\
\pi_{\R}\colon \mathcal{R}\to \Lambda^2T^*\otimes \R.
\end{gather*}
Since both the image of $\pi_{\gl^\perp}$ and the image of $\pi_\R$ contain $V_1^1\oplus\R$, it is possible to recover $\ric'$ by only considering these projections. Explicitly, we have 
\begin{equation}
  \label{eqn:projcurv}
D\tau + [\tilde\lambda,\tau]=\pi_{\gl^\perp}(R), \quad (d\lambda-\frac1n F(\tau,\tau)) \otimes \begin{pmatrix} I  & 0 \\ 0 & -I\end{pmatrix}= \pi_{\R}(R).
\end{equation}
Then
\begin{gather*}
\pi_{\gl^\perp}(R)=bv_3+hw_1, \\
\pi_\R(R)=(2h+k(1+\frac1n))w_2+(a+\frac2n(a+b))e^{n,n+1}\otimes (e^k\otimes e_k - e^{n+k}\otimes e_{n+k}). 
\end{gather*}
Thus \eqref{eqn:projcurv} gives
\begin{gather*}
\ric'(D\tau + [\tilde \lambda,\tau]) = b(n-2)e^n\odot e^{n+1} -2h(n-1)g,\\
\ric'((d\lambda-\frac1n F(\tau,\tau))\! \otimes\! \begin{pmatrix} I  & 0 \\ 0 & -I\end{pmatrix})=(a\!+\!\frac2n(a\!+\!b))e^n\!\odot\! e^{n+1}+(2h\!+\!k(1\!+\!\frac1n))g.
\end{gather*}

It is now straightforward to verify that the linear combination
\[\ric'=2\ric'(D\tau + [\tilde \lambda,\tau]) +n\ric'(D\tilde\lambda+\frac12[[\tau,\tau]]_\R)\]
is consistent with \eqref{eqn:ricprime}; the statement follows observing that for any two-form $\eta$, through the identification \eqref{eqn:identify11},
\[\ric'\biggl(\eta\otimes\begin{pmatrix} I  & 0 \\ 0 & -I\end{pmatrix}\biggr)=\eta^{1,1}.\qedhere\]
\end{proof}

It turns out that the explicit dependence on the minimal connection (i.e. the term $D\tau$) can be partly eliminated from the formula:
\begin{theorem}
\label{thm:ricci}
The Ricci tensor of an $\SL(n,\R)$-structure satisfies
\begin{equation*}
\begin{split}
\ric'&=2\ric'( \mathfrak{a}(\nabla\tau_3)+\lambda^{1,0}\wedge \tau_3) + 2(n-2)df_4^{1,1}+n(d\lambda)^{1,1}-2(n-1) f_4\wedge f_8\\
+&( 2n\Lambda(df_4)  -4(n-1)\langle f_4,f_8\rangle ) F-10F(\tau_1,\tau_5)+2F(\tau_1,\tau_6)-4F(\tau_2,\tau_5)\\
-&2F(\tau_2,\tau_6)+2\overline F(\tau_2,\tau_6)-2F(\tau_3,\tau_7+\tau_8)-2(n-1)F(\tau_4,\tau_7).
\end{split}
\end{equation*}
\end{theorem}
\begin{proof}
As $[\tilde \lambda,\tau_i]$ depends equivariantly on an element of $T^*\otimes W_i$, by equivariance, the only contribution of $[\tilde\lambda,\tau ]$ to $\ric'$ comes from $[\tilde\lambda,\tau_3+\tau_4+\tau_7+\tau_8]$. Since $\tau_3+\tau_4\in V^*\otimes \Lambda^2V$ and $\tau_7+\tau_8\in V\otimes \Lambda^2V^*$, 
\begin{equation}
\label{eqn:bracketlambdatau}
\ric'([\tilde\lambda,\tau])=\ric'(2\lambda\wedge (\tau_3+\tau_4)-2\lambda\wedge(\tau_7+\tau_8)).
\end{equation}
By \eqref{eqn:exteriorcovariant},
\[D(\tau_1+\tau_2) =-\partial(\tau+\tilde\lambda)\hook (\tau_1+\tau_2)\mod T^*\otimes (W_1+W_2);\]
and the same holds for $\tau_5+\tau_6$; by equivariance,
\[\ric'\!(D\tau) \!= \! \ric'\!(-\partial(\tau_5+\tau_6)\hook\!(\tau_1+\tau_2)
  -\partial(\tau_1+\tau_2)\hook (\tau_5+\tau_6) + D\tau_3+D\tau_4+D\tau_7+D\tau_8).\]
Since the Ricci tensor is symmetric, we can identify $\ric'$ with its projection on $V^*\otimes V$; using Lemma~\ref{lemma:ricprime},
\[\ric'\!=2\ric'\!(  D\tau_3+D\tau_4   -\partial(\tau_1+\tau_2)\hook\! (\tau_5+\tau_6)+2\lambda^{1,0}\wedge (\tau_3+\tau_4))+nd\lambda^{1,1}-F(\tau,\tau)^{1,1}.\]
Recalling that the contraction of $\partial(\tau_7+\tau_8)$ into $\tau_3+\tau_4$ is zero and \eqref{eqn:bracketlambda}, we obtain
\begin{align*}
\ric'(D\tau_3+D\tau_4)&=\ric'(\mathfrak{a}(\nabla\tau_3+\nabla\tau_4)-\partial(\tilde\lambda)\hook (\tau_3+\tau_4))\\
&=\ric'(\mathfrak{a}(\nabla\tau_3+\nabla\tau_4)-\lambda\wedge (\tau_3+ \tau_4) ).
\end{align*}
Therefore,
\begin{multline*}
\ric'=2\ric'( \mathfrak{a}(\nabla\tau_3+\nabla\tau_4) -\partial(\tau_1+\tau_2)\hook (\tau_5+\tau_6)+\lambda^{1,0}\wedge (\tau_3+\tau_4))\\
+n(d\lambda)^{1,1}-F(\tau,\tau)^{1,1}.
\end{multline*}
Writing $\tau_4$ as in \eqref{eqn:tau4tau8lambda},
\begin{align*}
\ric'(2\lambda^{1,0}\wedge\tau_4)&= 2\bigl(\lambda_i a_i e^k\otimes e^{n+k} +(n-2)\lambda^{1,0}\otimes a_ie^{n+i}\bigr)\\
&= 2(n-2)\lambda^{1,0}\otimes f_4 + 2\langle \lambda^{1,0},f_4\rangle F.
\end{align*}
In order to rewrite the term containing the covariant derivative of $\tau_4$, we may assume
$\nabla\tau_4 = a_{ij} e^j\otimes (e^k\otimes e_{ki});$
this implies
\[\ric(\mathfrak{a}(\nabla\tau_4) )= a_{ii}F + (n-2)a_{ij}e^j\otimes e^{n+i}= a_{ii}F +(n-2)\nabla f_4.\]
On the other hand we have
\begin{align*}
(df_4)^{1,1}& = \mathfrak{a}(\nabla f_4)^{1,1}-\partial(\tau_7+\tau_8)\hook f_4 + \lambda^{1,0}\wedge f_4\\
&= a_{ij}e^{j,n+i} +2d_{ijk}a_k e^{j,n+i} + \lambda^{1,0}\wedge f_4\\
&= a_{ij}e^{j,n+i} +F(\tau_7+\tau_8,\tau_4) + \lambda^{1,0}\wedge f_4,
\end{align*}
where
\[F(\tau_8,\tau_4)=\langle f_4,f_8\rangle F + f_4\wedge f_8.\]
In particular
\[\Lambda (df_4)=\frac1n a_{ii}+ (1-\frac1n)\langle f_4,f_8\rangle+\frac1n\langle \lambda^{1,0},f_4\rangle,\]
giving
\begin{multline*}
\ric(\mathfrak{a}(\nabla\tau_4) )=  ( n\Lambda(df_4) -(n-1)\langle f_4,f_8\rangle-\langle \lambda^{1,0},f_4\rangle) F \\
+(n-2)(df_4^{1,1}-F(\tau_7+\tau_8,\tau_4) - \lambda^{1,0}\wedge f_4).
\end{multline*}
Decomposing $\tau$ as in \eqref{eqn:tauindices}, we find
\begin{align*}
F(\tau_1+\tau_2,\tau_5+\tau_6)&=- 2a_{ijk}b_{hjk} e^i\wedge e^{n+h}, \\
\overline{F}(\tau_1+\tau_2,\tau_5+\tau_6)&=- 4a_{ijk}b_{ijh} e^k\wedge e^{n+h};
\end{align*}
then
\begin{multline*}
\ric(\partial(\tau_1+\tau_2)\hook(\tau_5+\tau_6))= 4(a_{jik}b_{khj}-a_{ijk}b_{khj})e^i\otimes e^{n+h}\\
=4F(\tau_1,\tau_5)-2F(\tau_1,\tau_6)+ F(\tau_2,\tau_5)-\overline F(\tau_2,\tau_6). 
\end{multline*}
By contrast,
\[F(\tau,\tau)=2F(\tau_1+\tau_2,\tau_5+\tau_6)+2F(\tau_3+\tau_4,\tau_7+\tau_8).\]
Summing up,
\begin{multline*}
\ric'=2\ric'( \mathfrak{a}(\nabla\tau_3)+\lambda^{1,0}\wedge \tau_3) + 2(n-2)(df_4-F(\tau_7+\tau_8,\tau_4)-\lambda^{1,0}\wedge f_4)\\
+2( n\Lambda(df_4) -(n-1)\langle f_4,f_8\rangle) F
-8F(\tau_1,\tau_5)+4F(\tau_1,\tau_6)-2F(\tau_2,\tau_5)+2\overline F(\tau_2,\tau_6)\\
+2(n-2)\lambda^{1,0}\wedge \tau_4+n(d\lambda)^{1,1}-2F(\tau_1+\tau_2,\tau_5+\tau_6)-2F(\tau_3+\tau_4,\tau_7+\tau_8),
\end{multline*}
from which we obtain the statement.
\end{proof}

As an $\SL(n,\R)$-module, $V^*\otimes V$ splits as the sum $V_1^1\oplus\R$; the two components of $\ric'$ in this decomposition can be written as $\ric'-sF$ and $s$, where $s$ denotes the scalar curvature
\[s=\frac1n\ric(e_i,e_{n+i}).\]
\begin{corollary}
\label{cor:s}
The scalar curvature of an $\SL(n,\R)$-structure is given by
\begin{multline*}
s=\frac{10}n\langle \tau_1,\tau_5\rangle -\frac2n \langle \tau_2,\tau_6\rangle
-\frac2n \langle \tau_3,\tau_7\rangle\\
+4(n-1)\Lambda df_4 - \frac{2(n-1)(2n-1)}{n}\langle f_4,f_8\rangle+n\Lambda (d\lambda).
\end{multline*}
\end{corollary}
\begin{proof}
By construction, $s=\Lambda \ric'$; in addition, the components which do not contain a copy of $\R$ give no contribution to the scalar curvature by equivariance. This gives
\begin{multline*}
s=(4n-4)\Lambda df_4-10\Lambda F(\tau_1,\tau_5)-2\Lambda F(\tau_2,\tau_6)+2\Lambda\overline F(\tau_2,\tau_6)\\
+n\Lambda(d\lambda)-2\Lambda F(\tau_3,\tau_7)-4(n-1)\langle f_4,f_8\rangle  +2(n-1)\Lambda( f_8\wedge f_4).
 \end{multline*}
A direct computation gives
\begin{align*}
\Lambda(F(\tau_1+\tau_2,\tau_5+\tau_6))&= -\frac1n \langle \tau_1,\tau_5 \rangle -\frac1n \langle \tau_2,\tau_6 \rangle,  \\
\Lambda(\overline{F}(\tau_1+\tau_2,\tau_5+\tau_6))&= -\frac2n \langle \tau_1,\tau_5 \rangle -\frac2n \langle \tau_2,\tau_6 \rangle,  \\
\Lambda(F(\tau_3+\tau_4,\tau_7+\tau_8))&= \frac1n \langle \tau_3,\tau_7 \rangle +\frac1n \langle \tau_4,\tau_8 \rangle,
\end{align*}
proving the statement.
\end{proof}

\begin{remark}
On a metric of neutral signature, the notion of ``positive'' scalar curvature is not meaningful: the pseudoriemannian metrics $g$ and $-g$ have the same Ricci tensor, but opposite scalar curvature. In our setup, this means that if we keep the splitting $V\oplus V^*$ but flip the sign of $F$, considering the $\SL(n,\R)$-structure determined by the adapted coframe
\[-e^1,\dotsc, -e^n,e^{n+1},\dotsc, e^{2n},\]
then the $\tau$ and the $\lambda$ stay the same, but $s$ changes its sign.
\end{remark}

\smallskip
The remaining part of the Ricci is given as follows. Denote by $\epsilon$ the symmetrization map
\[\epsilon\colon T^*\otimes T^*\to S^2T^*, \quad \eta\otimes \gamma \mapsto \eta\odot\gamma.\]
\begin{theorem}\label{thm:RicciS2}
The $\ric''$ component of the Ricci tensor of an $\SL(n,\R)$-structure satisfies
\begin{equation*}
 \begin{split}
 [\ric'']&_{S^2V} = \epsilon \Bigl((n-1)(\nabla f_4-f_4\otimes f_4+\lambda^{0,1}\otimes f_4) \\
&+\ric (\mathfrak{a}\nabla\tau_6-\partial(\tau_7)\hook(\tau_5 +\tau_6)-\partial(\tau_8)\hook \tau_6 
-\partial(\tau_3)\hook \tau_3+ 3\lambda^{1,0}\wedge\tau_6)^{0,2}\Bigr),\\
[\ric'']&_{S^2V^*} = \epsilon \Bigl((n-1)(\nabla f_8- f_8\otimes f_8-\lambda^{1,0}\otimes f_8) \\
&+\ric (\mathfrak{a}\nabla\tau_2-\partial(\tau_3)\hook(\tau_1 +\tau_2)-\partial(\tau_4)\hook \tau_2-\partial(\tau_7)\hook \tau_7- 3\lambda^{0,1}\wedge\tau_2)^{2,0}\Bigr). 
 \end{split}
\end{equation*}
\end{theorem}
\begin{proof}
In the case we consider the action of $\GL(n,\R)$; two representations are isomorphic if they are under $\SL(n,\R)$ and they have the same weight under $\R^*$.

The space $\Lambda^2T^*\otimes \Lambda^2T^*$ contains two copies of $S^2V$, contained in $\Sl(n,\R)\otimes \Sl(n,\R)^\perp$ and $\Sl(n,\R)^\perp\otimes \Sl(n,\R)$ respectively. Since the Riemann tensor is symmetric, this means that the $S^2V$ part of the Ricci is entirely determined by its component in $\Sl(n,\R)\otimes \Sl(n,\R)^\perp$, i.e.
\[[\ric(\Omega)]_{S^2V} = \epsilon (\ric([D\tau + [\tilde\lambda,\tau]]_{S^2V})).\]
The only $W_i\otimes W_j$ that contain a copy of $S^2V$ are 
\[W_5\otimes W_7,\ W_6\otimes W_7,\ W_6\otimes W_8,\ W_3\otimes W_3,\ W_4\otimes W_4.\]
Moreover, the two copies of $S^2V$ inside $T^*\otimes W_i$ are contained in $V\otimes W_4$  and $V^*\otimes W_6$. In consequence,
\begin{equation*}
\begin{split}
[D\tau]_{S^2V}& =
[\mathfrak{a}(\nabla\tau)-\partial(\tau)\hook\tau - \partial(\tilde\lambda)\hook\tau]_{S^2V}\\
&=
\bigl[\mathfrak{a}(\nabla\tau_4+\nabla\tau_6)-\partial(\tau_7)\hook(\tau_5 +\tau_6)-\partial(\tau_8)\hook \tau_6 \\
&\qquad -\partial(\tau_3)\hook \tau_3-\partial(\tau_4)\hook \tau_4-\lambda^{0,1}\wedge \tau_4 + \lambda^{1,0}\wedge\tau_6\bigr]_{S^2V}.
\end{split}
\end{equation*}
On the other hand \eqref{eqn:bracketlambdatau} gives
\[
[[\tilde\lambda,\tau]]_{S^2V} = [2\lambda^{0,1}\wedge \tau_4+2\lambda^{1,0}\wedge\tau_6]_{S^2V}.\]
Now
\[\ric(\partial(\tau_4) \hook\tau_4)= (n-1)a_ia_je^{n+i}\otimes e^{n+j} = (n-1)f_4\otimes f_4;\]
writing $a_{ij}e^{n+i}\otimes e^{n+j}$ for the $V\otimes V$ component of $\nabla f_4$, we obtain
\[\ric(\mathfrak{a}\nabla\tau_4)= (n-1)a_{ij}e^{n+i}\otimes e^{n+j}.\] 
Moreover
\[\ric(\lambda^{0,1}\wedge\tau_4)= (n-1)\lambda^{0,1}\otimes f_4,\]
giving the first formula in the statement. The second formula is obtained applying the symmetry $\sigma$ that interchanges $V$ and $H$.
\end{proof}

\section{Forms}
In this section we find formulae that express the intrinsic torsion and the Ricci curvature in terms of exterior derivatives, rather than exterior covariant derivatives. In particular we relate the intrinsic torsion to the exterior derivatives of the forms $\alpha\in\Lambda^{n,0},\ \beta\in\Lambda^{0,n},\ F\in\Lambda^{1,1}$.

First, we observe that $v\mapsto v\wedge\alpha$, $v\mapsto v\wedge\beta$ induce isomorphisms
\[\pi_{0,1}\colon \Lambda^{n,1}\to \Lambda^{0,1},\quad \pi_{1,0}\colon \Lambda^{1,n}\to \Lambda^{1,0}.\]
\begin{proposition}
\label{prop:forms}
The intrinsic torsion of an $\SL(n,\R)$-structure determines $dF$, $d\alpha$ and $d\beta$ via
\begin{gather*}
(d F)^{3,0} = -\partial(\tau_1)\hook F, \qquad (dF)^{2,1}= -\partial(\tau_7)\hook F-2f_8\wedge F,\\
(d F)^{0,3} = -\partial(\tau_5)\hook F, \qquad (dF)^{1,2}= -\partial(\tau_3)\hook F -2f_4\wedge F,\\
(d\alpha)^{n,1}= -(n\lambda^{0,1} + (n-1)f_4)\wedge\alpha, \qquad 
(d\alpha)^{n-1,2}=  -\partial(\tau_5+\tau_6)\hook\alpha ,\\
(d\beta)^{1,n}=  (n\lambda^{1,0} - (n-1)f_8)\wedge \beta , \qquad 
(d\beta)^{2,n-1}=   -\partial(\tau_1+\tau_2)\hook\beta.
\end{gather*}
Conversely,
\begin{gather*}
f_4=-\frac12\left(\Lambda (dF)\right)^{0,1},\qquad f_8=-\frac12\left(\Lambda (dF)\right)^{1,0},\\
\tau_1 = \frac16 e^i\otimes e_i\hook (d F)^{3,0},\qquad\tau_5 = -\frac16 e^{n+i}\otimes e_{n+i}\hook (d F)^{0,3},\\
\tau_2 =\frac12\langle ((d\beta)^{2,n-1}+\partial(\tau_1)\hook\beta),e^{n+j,n+k}\wedge (e_{i}\hook\alpha)\rangle e^{i}\otimes e^{j,k},\\
\tau_6 =\frac12\langle ((d\alpha)^{n-1,2}+\partial(\tau_5)\hook\alpha),e^{jk}\wedge (e_{n+i}\hook\beta)\rangle e^{n+i}\otimes e^{n+j,n+k},\\
\tau_3 = -\frac12  e^i\otimes e_i\hook (dF)^{1,2}-\tau_4,\quad
\tau_7 = \frac12  e^{n+i}\otimes e_{n+i}\hook (dF)^{2,1}-\tau_8,\\
\lambda^{0,1} = -\frac1n \pi_{0,1}(d \alpha)^{n,1}-\frac{n-1}{n}f_4,\qquad\lambda^{1,0} = \frac1n \pi_{1,0}(d \beta)^{1,n}+\frac{n-1}{n}f_8.
\end{gather*}
\end{proposition}
\begin{proof}
The usual formula \eqref{eqn:exteriorcovariant} gives
\[d\alpha = -\partial(\tau+\tilde\lambda)\hook\alpha = -\partial(\tau_3+\tau_4+\tau_5+\tau_6)\hook\alpha - n\lambda^{0,1}\wedge\alpha.\]
However, $-\partial(\tau_3)\hook \alpha$ is zero because it lies in $\Lambda^{n,1}\cong V$ but $W_3$ is not isomorphic to $V$; in addition: 
\begin{gather*}\partial(\tau_4)\hook\alpha= -\sum_{i\neq k} a_i e^{k,n+i}\otimes e_k\hook\alpha =  \sum_{i\neq k}a_ie^{n+i}\wedge \alpha=(n-1)f_4\wedge \alpha\\
d\alpha = -\partial(\tau_5+\tau_6)\hook\alpha -(n\lambda^{0,1} +(n-1)f_4)\wedge\alpha
\end{gather*}
where we have used  \eqref{eqn:pag7alto}.
Similarly, $\partial(\tau_8)\hook\beta= (n-1)f_8\wedge \beta,$
\[d\beta = -\partial(\tau_1+\tau_2)\hook\beta +(n\lambda^{1,0}-(n-1)f_8)\wedge\beta.\]
For $F$ we compute
\[dF= -\partial(\tau+\tilde\lambda)\hook F=-2(f_4+f_8)\wedge F-\partial(\tau_1+\tau_3+\tau_5+\tau_7)\hook F\]
where we have used that $\partial(\tau_2+\tau_6)\hook F$ is zero because $W_2$ and $W_6$ are not isomorphic to any subspace of $\Lambda^3=\Lambda^{3,0}\oplus\Lambda^{2,1}\oplus\Lambda^{1,2}\oplus\Lambda^{0,3}$. By projecting this last formula on the different types of $(p,q)$-forms we get the statement.

We now prove the inverse formulae. From the last equation, we immediately get 
\[-\frac12\Lambda (dF)=f_4+f_8\]
and the equations for $\lambda^{1,0}$ and $\lambda^{0,1}$ are obvious.

If we set $\tau_1=e^i\otimes e^{jk}+e^j\otimes e^{ki}+e^k\otimes e^{ij}$; then
\[(d F)^{3,0}=-\partial(\tau_1)\hook F = 6 e^{ijk}\]
and it follows that 
\[\frac16 e^h\otimes e_h\hook (d F)^{3,0}=\tau_1.\]
Setting $\tau_3+\tau_4= e^i\otimes e^{n+j,n+k}$ we obtain
\[(d F)^{1,2}=-\partial(\tau_3+\tau_4)\hook F =-2 e^{i,n+j,n+k},\]
thus
\[\tau_3 + \tau_4 = -\frac12 e^i\otimes e_i\hook (dF)^{1,2}.\]
Finally if $\tau_2= a_{ijk}e^i\otimes e^{jk}$, for a fixed tensor $e^{n+h,n+l}\otimes e_m$ we have
\begin{align*}
\langle -\partial(\tau_2)&\hook\beta, e^{n+h,n+l}\otimes e_m\hook\alpha \rangle \\
&= \langle - a_{ijk} e^{ij}\otimes e_{n+k}\hook\beta+a_{ijk} e^{ik}\otimes e_{n+j}\hook\beta, e^{n+h,n+l}\otimes e_m\hook\alpha\rangle \\
&= -2a_{hlm}+2a_{lhm}= 2a_{mhl},
\end{align*}
since $a_{ijk}+a_{jki}+a_{kij}=0$ and $a_{ijk}=-a_{ikj}$. Thus, 
\[\tau_2=  \frac12\langle -\partial(\tau_2)\hook\beta, e^{n+j,n+k}\wedge e_i\hook\alpha \rangle e^i\otimes e^{jk}.\]
The remaining equations are proved in the same way.
\end{proof}

We can relate each component of the intrinsic torsion as we have done for $\tilde{\lambda}$, related to the 1-form $\lambda$, and $\tau_4$,  $\tau_8$  (see \eqref{eqn:pag7alto}). Indeed, the component $\tau_1=a_{ijk}(e^i\otimes e^{jk}+e^j\otimes e^{ki}+e^k\otimes e^{ij})$ can be associated to the $(3,0)$-form $a_{ijk}e^{ijk}$, and analogously for $\tau_5$; we set
\[f_3=-\partial (\tau_3)\hook F,\qquad f_7=-\partial (\tau_7)\hook F.\]
Finally, if $\tau_2=a_{ijk} e^{i}\otimes e_{n+j,n+k}$, we set 
\[f_2=a_{ijk}  (e_{n+j,n+k}\hook \beta)\wedge e^{i}\in\Lambda^{1,n-2},\]
and if $\tau_6=b_{ijk} e^{n+i}\otimes e_{jk}$ we set 
$f_6=b_{ijk}  (e_{jk}\hook \alpha)\wedge e^{n+i}\in\Lambda^{n-2,1}$. We use the following convention: for any $p$-form $\sigma$, $p\geq 2$  and any bi-vector $e_{jk}$ the $(p-2)$-form $e_{jk}\hook \sigma$ is defined by
\[(e_{jk}\hook \sigma)(X_1,\dots,X_{p-2})=\sigma (e_j, e_k, X_1,\dots,X_{p-2});\]
equivalently, $e_{jk}\hook \sigma=e_k\hook(e_{j}\hook \sigma)$. 

We can then restate the equations of the Ricci curvature by expressing the $\nabla \tau_i$ in terms of the exterior derivative of the forms $f_i$.  To this purpose, we identify $V\otimes V$ with $\Lambda^{n-1,1}$ through
\begin{equation}
 \label{eqn:VotimesVwithLambdan1}
v\otimes w\mapsto (v\hook\alpha)\wedge w,
 \end{equation}
enabling us to identify a subspace of $\Lambda^{n-1,1}$ isomorphic to $S^2V$. Similarly, to obtain a subspace isomorphic to $S^2V^*$ we identify $V^*\otimes V^*$ with $\Lambda^{1,n-1}$ via the isomorphism $v \otimes w \mapsto (v\hook\beta)\wedge w$.

\begin{lemma}\label{lemma:f2f6equations}
The following equations hold for $f_2$ and $f_6$:
\begin{align*}
[\ric(\mathfrak{a}(\nabla \tau_6)-\partial(\tau_7+\tau_8)\hook\tau_6)+3\lambda^{1,0}\wedge \tau_6]_{S^2V}&= [df_6+n\lambda^{1,0}\wedge f_6]_{S^2V},\\
[\ric(\mathfrak{a}(\nabla \tau_2)-\partial(\tau_3+\tau_4)\hook\tau_2)-3\lambda^{0,1}\wedge \tau_2]_{S^2V^*}& = [df_2-n\lambda^{0,1}\wedge f_2]_{S^2V^*}
\end{align*}
and the following identities hold for $f_4$ and $f_8$:
\begin{equation*}
\begin{alignedat}{3} [\nabla f_4- f_4\otimes f_4+\lambda^{0,1}\otimes f_4]_{S^2V}&=[ (-1)^{n-1}&&d(f_4\hook\alpha)]_{S^2V} \\
& &&+ \frac12n\lambda^{0,1}\odot f_4 + (n-2)f_4\otimes f_4,\\
  [\nabla f_8- f_8\otimes f_8-\lambda^{1,0}\otimes f_8]_{S^2V^*} &=[ (-1)^{n-1}&&d(f_8\hook\beta)]_{S^2V^*}\\
 & &&- \frac12n\lambda^{1,0}\odot f_8 + (n-2)f_8\otimes f_8.
 \end{alignedat}
\end{equation*}
\end{lemma}
\begin{proof}
We prove the first equation. As usual, we start from the following:
\[d f_6 = \mathfrak{a}(\nabla f_6) -\partial (\tau + \tilde{\lambda})\hook f_6.\]
If $\nabla \tau_6=a_{ijkh} e^i\otimes  e^{n+j}\otimes e_{kh} +b_{ijkh} e^{n+i}\otimes  e^{n+j}\otimes e_{kh}$,
then
\[ \nabla f_6=a_{ijkh} e^i\otimes(e_{kh}\hook\alpha)\wedge e^{n+j} + b_{ijkh} e^{n+i} \otimes(e_{kh}\hook\alpha)\wedge e^{n+j}.\]
We are interested in the $(n-1,1)$-component of $d f_6$, and more precisely in the $S^2(V)$ part. We get:
\[[df_6]^{n-1,1}=[\mathfrak{a}(\nabla f_6) -\partial (\tau_7 +\tau_8 + \tilde{\lambda})\hook f_6]^{n-1,1}.\]
Therefore,
\begin{equation*}
 \begin{aligned}
[df_6]_{S^2V}&=[a_{ijkh} e^{i}\wedge  (e_{kh}\hook\alpha)\wedge e^{n+j} - (n-3)\lambda^{1,0}\wedge f_6-\partial(\tau_7 +\tau_8)\hook f_6]_{S^2V}\\
&=[2a_{ijki}   (e_{k}\hook\alpha)\wedge e^{n+j} - (n-3)\lambda^{1,0}\wedge f_6-\partial(\tau_7 +\tau_8)\hook f_6]_{S^2V}\\
&=a_{ijki}   e^{n+k}\odot e^{n+j} - [(n-3)\lambda^{1,0}\wedge f_6+\partial(\tau_7 +\tau_8)\hook f_6]_{S^2V},
\end{aligned}
\end{equation*}
where we have used the identification \eqref{eqn:VotimesVwithLambdan1}. On the other hand, we have:
\[\ric (\mathfrak{a}(\nabla\tau_6)) = \ric(a_{ijkh} e^{i,n+j}\otimes (e^{n+k}\otimes e_h-e^{n+h}\otimes e_k))
=2a_{ijki} e^{n+j}\otimes e^{n+k}.
\]
Writing $\tau_6=b_{ijk} e^{n+i}\otimes e_{jk}$ and  $\tau_7 +\tau_8= d_{ijk} e^{n+i}\otimes e^{jk}$,
we compute
\[[\ric(-\partial(\tau_7 +\tau_8)\hook\tau_6)]_{S^2V}= 2b_{khj} d_{ijk} e^{n+i}\odot e^{n+h}.\]
Moreover,
\[
-\partial(\tau_7 +\tau_8)\hook f_6 = 4b_{khj} d_{ijk}  e_{h}\hook \alpha \wedge e^{n+i},\]
and the $S^2(V)$ component is $2b_{khj} d_{ijk}  e^{n+h}\odot e^{n+i}$. Finally, note that
\[[\lambda^{1,0}\wedge f_6]_{S^2V} = \lambda_k b_{ijk}e^{n+j}\odot e^{n+i} = [\ric(\lambda^{1,0}\wedge \tau_6)]_{S^2V},\]
which concludes the proof of the first equation. The second one is proved in a similar way.

Consider the $(n-1,0)$-form $f_4\hook \alpha = a_ie_i\hook\alpha$. If $\nabla\tau_4=a_{ij} e^{n+i}\otimes e^k\otimes e_{kj}\allowbreak+b_{ij} e^i\otimes e^k\otimes e_{kj}$ then 
\[\nabla (f_4\hook\alpha)=a_{ij}e^{n+i}\otimes e_j\hook\alpha + b_{ij}e^i\otimes e_j\hook\alpha.\]
As usual, $d (f_4\hook \alpha)=\mathfrak{a}(\nabla(f_4\hook\alpha))-\partial(\tau+\tilde{\lambda})\hook(f_4\hook\alpha)$, but we are interested in the $(n-1,1)$ part, and more precisely in the $S^2(V)$ component. We obtain: 
\begin{equation*}
\begin{aligned}
\left[d (f_4\hook \alpha)\right]_{S^2(V)}&=[\mathfrak{a}\!\left(\nabla(f_4\hook\alpha)\right)-\partial(\tau_4)\hook (f_4\hook\alpha)-(n\!-\!1)\lambda^{0,1}\!\wedge\!(f_4\hook\alpha)]_{S^2(V)}\\
= (-1)^{n-1}[a_{ij}e_j\hook\alpha&\!\wedge\! e^{n+i}-(n\!-\!1)a_j e_j\hook\alpha\!\wedge\! a_i e^{n+i}-(n\!-\!1)a_j e_j\hook\alpha\!\wedge\!\lambda^{0,1}]_{S^2(V)}\\
&= (-1)^{n-1}[\nabla f_4 -(n-1) f_4 \otimes f_4-(n-1)f_4\otimes\lambda^{0,1}]_{S^2(V)}.
 \end{aligned}
\end{equation*}
This ends the proof of the equation involving $f_4$. A similar argument proves the equation for $f_8$.
\end{proof}
\begin{lemma}\label{lemma:ric'nonabla} 
The $V^1_1$-Ricci part of $\tau_3$ and $\tau_7$ can be related to $f_3$ and $f_7$ via the following equations:
\begin{gather*} 
[\ric(\mathfrak{a}(\nabla\tau_3)+\lambda^{1,0}\wedge \tau_3)]_{V_1^1} = 
\frac{2-n}{2}[\Lambda(df_3+\partial(\tau_7+\tau_8)\hook f_3)]_{V_1^1},\\
[\ric(\mathfrak{a}(\nabla\tau_7)-\lambda^{0,1}\wedge \tau_7)]_{V_1^1} = 
\frac{2-n}{2}[\Lambda(df_7+\partial(\tau_3+\tau_4)\hook f_7)]_{V_1^1}.
\end{gather*}
 \end{lemma}
\begin{proof}
Let $h_3$ be the equivariant map
\[h_3\colon \Lambda^{1,1}_0\to T^*\otimes W_3, \qquad e^{i,n+j}\mapsto  \sum_k e^k\otimes  e^i\otimes e^{n+k,n+j}-\frac1n (e^i\otimes e^k\otimes e^{n+k,n+j}),\]
and suppose that $\nabla \tau_3 =h_3(e^{n,n+1})$; then 
\[\nabla f_3 = 2e^k\otimes e^{n+k,n,n+1} +\frac 2n e^n\otimes F\wedge e^{n+1}.\]
Using \eqref{eqn:exteriorcovariant} we have
\begin{equation*}
\begin{aligned}
d(f_3) &=\mathfrak{a}(\nabla f_3)-\partial(\tau+\tilde\lambda)\hook f_3\\
&= 2\frac{n+1}n F\wedge e^{n,n+1}-\partial(\tau)\hook f_3+\lambda\wedge f_3,
\end{aligned} 
\end{equation*} 
and an easy computation shows:
\[\Lambda(df_3+\partial(\tau_7+\tau_8)\hook f_3-\lambda^{1,0}\wedge f_3)^{2,2} =  2\frac{n+1}n e^{n,n+1}.\]
On the other hand we have:
\[\ric(\mathfrak{a}(\nabla\tau_3))= \frac{(2-n)(n+1)}n e^n\otimes e^{n+1};\]
thus,
\[[\ric(\mathfrak{a}(\nabla\tau_3))]_{V^1_1} = \frac{2-n}{2}\Lambda(df_3+\partial(\tau_7+\tau_8)\hook f_3-\lambda^{1,0}\wedge f_3)^{2,2}. \]
By the same token, if $\lambda^{1,0}\otimes \tau_3 = h_3(e^{n,n+1})$, we see that
\[\lambda^{1,0}\wedge f_3=2\frac{n+1}n F\wedge e^{n,n+1},\]
so
\[[\ric (\lambda^{1,0}\wedge \tau_3)]_{V_1^1} = \frac{(2-n)(n+1)}n  \frac{n}{2(n+1)} \Lambda(\lambda^{1,0}\wedge f_3).\]
Summing up, we get
\[[\ric(\mathfrak{a}(\nabla\tau_3)+\lambda^{1,0}\wedge \tau_3)]_{V_1^1} = 
\frac{2-n}{2}\Lambda(df_3+\partial(\tau_7+\tau_8)\hook f_3)^{2,2},\]
which proves the first of the two equations. By considering the equivariant map
\[h_7\colon \Lambda^{1,1}_0\to T^*\otimes W_7, \quad e^{i,n+j}\mapsto  \sum_k e^{n+k}\otimes e^{n+i}\otimes e^{k,j}-\frac1n (e^{n+i}\otimes e^{n+k}\otimes e^{k,j})\]
and using a similar procedure, the other equation  follows. 
\end{proof}

Using Lemmas \ref{lemma:f2f6equations} and \ref{lemma:ric'nonabla} we are able to restate Theorems \ref{thm:ricci} and \ref{thm:RicciS2}, expressing the Ricci curvature of the Levi-Civita connection without using the covariant derivative $\nabla\tau$. Note that the projection of $\ric ( -\partial(\tau_7)\hook\tau_5 -\partial(\tau_3)\hook \tau_3)$ on tensors of type $(2,0)$ is  now redundant and we can drop it; the same happens for the $S^2V^*$  component of the Ricci tensor.
\begin{theorem}\label{thm:restateRicViaForms}
The Ricci tensor of an $\SL(n,\R)$-structure satisfies:
\begin{equation*}
\begin{split}
\ric'&=(2-n)[\Lambda(df_3+\partial(\tau_7+\tau_8)\hook f_3)]_{V^1_1}+ 2(n-2)df_4^{1,1} -2(n-1) f_4\wedge f_8\\
&\!\!\!+( 2n\Lambda(df_4)  -4(n-1)\langle f_4,f_8\rangle ) F-10F(\tau_1,\tau_5)+2F(\tau_1,\tau_6)-4F(\tau_2,\tau_5)\\
&\!\!\!-2F(\tau_2,\tau_6)+2\overline F(\tau_2,\tau_6)-2F(\tau_3,\tau_7+\tau_8)-2(n-1)F(\tau_4,\tau_7)+n(d\lambda)^{1,1},\\
[\ric'']&_{S^2V} = \epsilon \Bigl((n-1)([(-1)^{n-1}d(f_4\hook\alpha)]_{S^2V} + \frac12n\lambda^{0,1}\odot f_4 + (n-2)f_4\otimes f_4 )\\
&+[df_6+n\lambda^{1,0}\wedge f_6]_{S^2V}+\ric ( -\partial(\tau_7)\hook\tau_5 -\partial(\tau_3)\hook \tau_3)\Bigr),\\
[\ric'']&_{S^2V^*}\!=\! \epsilon \Bigl((n-1)([ (-1)^{n-1}d(f_8\hook\beta)]_{S^2V^*} - \frac12n\lambda^{1,0}\odot f_8 + (n-2)f_8\otimes f_8) \\
&+[df_2-n\lambda^{0,1}\wedge f_2]_{S^2V^*}+\ric (-\partial(\tau_3)\hook\tau_1 -\partial(\tau_7)\hook \tau_7)\Bigr).
\end{split}
\end{equation*}
\end{theorem}

\begin{remark}
In the notation of \cite{HarveyLawson}, a $D$-valued $(n,0)$-form has the form
\[\Phi=(\overline e p + e q)(\overline e \alpha+ e\beta) = \overline{e}(p\alpha) + e(q\beta) .\]
Assume that the structure is paracomplex; then this is holomorphic if and only if it is closed; this is equivalent to 
\[d(\log p)^{0,1} = n\lambda^{0,1}+(n-1)f_4, \quad d(\log q)^{1,0} = -n\lambda^{1,0}+(n-1)f_8.\]
By \cite{HarveyLawson}, in the parak\"ahler case the Ricci form is given by
\[\ric=-\tau \partial\overline{\partial} \log \abs{\Phi}^2 = -\partial_+\partial_-\log \abs{\Phi}^2 .\]
In our case, $\abs{\Phi}^2=pq$, hence
\begin{align*} 
-\partial_+\partial_-  pq &= -\partial_+(n\lambda^{0,1}-(n-1)f_4))+ \partial_-(-n\lambda^{1,0}+(n-1)f_8) \\
&= (-nd\lambda + (n-1)(-df_4+df_8))^{1,1}.
\end{align*}
This is one component in our expression for the Ricci which corresponds to the curvature of the canonical bundle. In the parak\"ahler case, i.e. $\tau=0$, the expressions coincide.
\end{remark}

As an application, we recover a result of \cite{IvanovZamkovoy:parahermitian}, asserting that six-dimensional nearly parak\"ahler structures are Einstein. Recall that an almost parak\"ahler manifold is called nearly parak\"ahler if $\nabla K$ is skew-symmetric in the first two indices; in our language, this means that the intrinsic torsion lies in $W_1+W_5$ (see e.g. \cite{GadeaMasque}). From the paracomplex point of view, this means that under some identification $V\cong H^*$, the two components of the Nijenhuis tensor lie in
\[\Lambda^3V^*+\Lambda^3H^*;\]
i.e.,  the Nijenhuis tensor is totally skew-symmetric; this is the condition that ensures the existence of a connection with skew-symmetric torsion \cite{IvanovZamkovoy:parahermitian}. 
With this definition, a parak\"ahler manifold is nearly parak\"ahler; we shall say a structure is \emph{strictly nearly parak\"ahler} if $\tau_1\neq0$ at each point. It was shown in \cite{SchaferSchulteHengesbach} that six-dimensional strictly nearly parak\"ahler structures can be characterized in terms of differential forms and exterior derivatives. In our language, we obtain:
\begin{corollary}
\label{cor:nk}
On a connected six dimensional manifold $M$, an almost parahermitian structure with $\tau_1\neq0$ is strictly nearly parak\"ahler if and only if it has a reduction to $\SL(3,\R)$ such that, up to rescaling by a constant,
either
\begin{equation*}
dF=\alpha + \beta,\quad  d(\alpha-\beta)=-\frac13F^2
\end{equation*}
or
\begin{equation*}
 dF=\alpha,\quad  d\beta=\frac16F^2.
\end{equation*}
The metric is Einstein with $s=-\frac5{18}$ in the former case, and Ricci-flat in the latter.
\end{corollary}
\begin{proof}
It is clear from Proposition~\ref{prop:forms} that the two situations correspond to the intrinsic torsion classes $W_1+W_5$ or $W_1$.

Conversely, if the intrinsic torsion lies in $W_1+W_5$, with $\tau_1\neq0$, necessarily $dF$ has a component of type $(3,0)$; we can choose the reduction so that it equals $\alpha$. Then
\begin{equation*}
dF=\alpha + k\beta,\quad  (d\beta)^{2,2}=\frac16F^2,
\end{equation*}
for some function $k$. If $k$ is not zero at some point, we have
\[(d\alpha)^{3,1}=0,\quad (d\alpha)^{2,2}=-\frac{k}6F^2,\quad (d\beta)^{1,3}=-d(\log k)\wedge \beta.\]
This implies that $d\lambda = 0$. Therefore, the metric is Einstein with curvature $-\frac{5k}{18}$; up to rescaling we can assume $k\equiv 1$.

Otherwise $k$ is identically zero; this implies that
\[0=d^2\beta =d(3\lambda\wedge\beta + \frac16F^2)= 3d\lambda\wedge \beta - \frac12\lambda\wedge  F^2;\]
since $\lambda$ has type $(1,0)$, this implies that $\lambda=0$. Hence, the metric is Ricci-flat.
\end{proof}
The six-dimensional case is also special because a nearly parak\"ahler $6$-manifold gives rise to a nearly parallel $\Gtwo^*$-structure on a suitable warped product (\cite{CortesLeistner}).

\section{The Bott-Chern class}
Whilst the component $\tau$ of the intrinsic torsion depends only on the $\GL(n,\R)$-structure, the component $\lambda$ only depends on the $\SL(n,\R)\times \SL(n,\R)$-structure. In analogy with the complex case, it is possible to define an invariant that does not depend on a metric, playing the same role as the first Chern class.

Recall that the (para)Bott-Chern class cohomology spaces are defined as
\[H^{p,q}_{BC} =\frac{ \ker d\colon\Omega^{p,q}\to \Omega^{p+1,q}\oplus\Omega^{p,q+1} }{\im dd^c\colon \Omega^{p-1,q-1}\to \Omega^{p,q}},\]
where $d^c=K\circ d\circ K$. They are not finite-dimensional in general \cite{AngellaRossi:cohomology2012}. 

\begin{proposition}
\label{prop:modlambda}
Let $(M,K)$ be a paracomplex manifold that admits a reduction to $\SL(n,\R)$, say $(F,\alpha,\beta)$. If $(\tilde F,\tilde\alpha, \tilde\beta)$ is another reduction to $\SL(n,\R)$ with 
\[\tilde \alpha = e^{n(h+k)} \alpha, \quad \tilde\beta=e^{n(h-k)}\beta,\]
then its intrinsic torsion satisfies
\[\tilde\lambda = \lambda + d^ch -dk.\]
In particular, the Bott-Chern class $[d\lambda]$ only depends on $K$.
\end{proposition}
\begin{proof}
Straightforward from Proposition \ref{prop:forms}:
\begin{gather*}
-n\tilde\lambda\wedge\tilde \alpha = d(\tilde\alpha ) = n(dh+dk)\wedge \tilde\alpha -n \lambda^{0,1}\wedge \tilde\alpha\\
n\tilde\lambda\wedge\tilde \beta = d(\tilde\beta ) = n(dh-dk)\wedge \tilde\beta +n\lambda^{1,0}\wedge \tilde\beta.\qedhere\\ 
\end{gather*}
\end{proof}

\begin{proposition}
Let $(M,K)$ be a paracomplex manifold which admits a reduction to $\SL(n,\R)$. If $(M,K)$ admits a compatible  Ricci-flat parak\"ahler metric, then $[d\lambda]=0$ in $H^{1,1}_{BC}$.
\end{proposition}
\begin{proof}
Theorem~\ref{thm:restateRicViaForms} gives $\ric'=d\lambda$. Since $[d\lambda]$ only depends on $K$, if a Ricci-flat parak\"ahler metric exists necessarily $d\lambda=0$. 
\end{proof}
We do not know whether an analogue of the Aubin-Yau theorem holds~\cite{Aubin,Yau}: Proposition~\ref{prop:modlambda} only tells us what the volume of a potential Ricci-flat parak\"ahler metric should be, but does not guarantee existence. 

On the other hand, we have the following:
\begin{proposition}
Let $(F,\alpha,\beta)$ be a paracomplex $\SL(n,\R)$-structure on $M$. Then $M$ admits a parak\"ahler-Einstein metric with $s\neq0$ compatible with the same paracomplex structure if and only if there exists a function $h$ such that
\[(d\lambda + dd^ch)^n = e^{2nh}F^n.\]
\end{proposition}
\begin{proof}
Any compatible $\SL(n,\R)$-structure has the form $(\tilde F,\tilde\alpha, \tilde\beta)$, with 
\[\tilde \alpha = e^{n(h+k)} \alpha, \quad \tilde\beta=e^{n(h-k)}\beta.\]
Since $k$ does not affect the metric, it is no loss of generality to assume $k=0$. Then from Proposition \ref{prop:forms} we have:
\[d\tilde\lambda = d \lambda + dd^ch.\]
Set $\tilde F=d\tilde\lambda$; then $(\tilde F,\tilde \alpha,\tilde\beta)$ defines an $\SL(n,\R)$-structure if and only if
\[\tilde F^n = e^{2nh}F^n.\]
If this condition holds, the structure is automatically parak\"ahler and Einstein, because $\tilde F$ is closed and because of Theorem~\ref{thm:ricci}.

By the same token, a compatible parak\"ahler-Einstein metric necessarily satisfies $\tilde F=d\tilde\lambda$.
\end{proof}
This condition is the analogue of the Monge-Amp\`ere equation in complex geometry. A striking difference is that the volume form is automatically exact; this means that it makes no sense to assume the manifold is compact. In fact,  \emph{compact} Einstein parak\"ahler manifolds are necessarily Ricci-flat:
\begin{proposition}
\label{prop:pkeinsteincompact}
If $M$ is a compact, parak\"ahler manifold, then 
\[\int_M s F^n =0.\] 
A left-invariant parak\"ahler structure on a unimodular Lie group has $s=0$.
\end{proposition}
\begin{proof}
By Corollary~\ref{cor:s}, $d\lambda\wedge F^{n-1} = \frac1n sF^n$; this form is exact because $F$ is symplectic. If $M$ is compact, Stokes' theorem gives the statement; the same argument applies to the case of a unimodular Lie group,  since there are no exact invariant volume forms.
\end{proof}
\begin{remark}
More generally, this argument applies to \emph{balanced} almost parak\"ahler structures, i.e. those where $F^{n-1}$ is closed; then, in the case of an invariant structure on a unimodular Lie group, $\Lambda(d\lambda)=0$ and $\tau_4=0=\tau_8$. Therefore, the expression for the scalar curvature reduces to 
\[
s=\frac{10}n\langle \tau_1,\tau_5\rangle -\frac2n \langle \tau_2,\tau_6\rangle-\frac2n \langle \tau_3,\tau_7\rangle.
\]
\end{remark}

\smallskip
Another difference with the (compact) complex case is that an Einstein para\-k\"ahler manifold with $s\neq0$ can have $[d\lambda]=0$; in fact, even $H^{1,1}_{BC}=0$ as in Example \ref{ex:Einstein}. On a paracomplex manifold $M$, decompose $d$ as $\partial_++ \partial_-$,
\[\partial_+\colon \Omega^{p,q}\to \Omega^{p+1,q}, \quad \partial_-\colon \Omega^{p,q}\to \Omega^{p,q+1}.\]
We can then define the (para)Dolbeault cohomology
\[H^{p,q}_+ = \frac{\ker \partial_+\colon \Omega^{p,q}\to \Omega^{p+1,q} }{\im \partial_+\colon \Omega^{p-1,q}\to \Omega^{p,q}},\quad H^{p,q}_- = \frac{\ker \partial_-\colon \Omega^{p,q}\to \Omega^{p,q+1} }{\im \partial_-\colon \Omega^{p,q-1}\to \Omega^{p,q}}.\]
\begin{theorem}
\label{thm:aubinyau_dei_poveri}
Let $(M,K)$ be a paracomplex manifold with 
\[H^1(M)=0=H^2(M)=H^{1,0}_+(M)=H^{0,1}_-(M).
\]
Then $H^{1,1}_{BC}$ is zero.
\end{theorem}
\begin{proof}
Let $\mathcal{Z}^{1,1}$ be the sheaf of closed $(1,1)$-forms, and let $\mathcal{K}$ be the sheaf of $dd^c$-closed functions; we have a short exact sequence of sheaves
\[0\to \mathcal{K}\to \mathcal{A}^0\xrightarrow{dd^c} \mathcal{Z}^{1,1}\to 0.\]
This gives
\[0\to H^0(\mathcal{K})\to H^0(\mathcal{A}^0)\xrightarrow{dd^c} H^0( \mathcal{Z}^{1,1})\to H^1(\mathcal{K})\to 0,\]
where we have used the fact that $\mathcal{A}^0$ is acyclic. Under the present assumptions, the $dd^c$ lemma is equivalent to $H^1(\mathcal{K})=0$.

Now consider the sheaves
\[\mathcal{Z}^{p,q}_\pm =\ker \partial_{\pm}\colon  \mathcal{A}^{p,q}\to \mathcal{A}^{p+1+q} ;\]
then the sequence
\[0\to\mathcal{Z}^0\to \mathcal{Z}^{0,0}_+\oplus \mathcal{Z}^{0,0}_-\to \mathcal{K}\to 0\]
is exact. Indeed, assume  the first and second de Rham cohomology of an open set $U$ is zero; then
\[\gamma\in \mathcal{K}(U)\implies dd^{c}\gamma=0 \implies d^c\gamma = df\implies \partial_+(\gamma-f)=0,\  \partial_-(\gamma+f)=0;\]
thus $\gamma\in \mathcal{Z}^{0,0}_+(U)\oplus \mathcal{Z}^{0,0}_-(U)$. 

Taking now the associated long exact sequence, and using the fact that $H^1(M)$ and $H^2(M)$ are zero, we deduce
\[H^1(\mathcal{K})=H^1( \mathcal{Z}^{0,0}_+)\oplus H^1( \mathcal{Z}^{0,0}_+).\]

Since a paracomplex manifold is locally a product, the following sequence is exact:
\[0\to  \mathcal{Z}^{0,0}_+\to \mathcal{A}^0\xrightarrow{\partial_+}\mathcal{Z}^{1,0}_+\to 0,\]
giving the exact sequence in cohomology
\[H^0(\mathcal{A}^0)\xrightarrow{\partial_+}H^0(\mathcal{Z}^{1,0}_+)\to H^1(\mathcal{Z}^{0,0}_+)\to 0.\]
This shows that  $H^1(\mathcal{Z}^{0,0}_+)=H^{1,0}_+(M)=0$, and a similar result applies to the minus component. Thus, $H^1(\mathcal{K})=0$, which is what we had to prove.
\end{proof}

\begin{example}\label{ex:Einstein}
On the solvable Lie group $G=\mathrm{Aff}_+(\R)\times \mathrm{Aff}_+(\R)\times \mathrm{Aff}_+(\R)$, consider a left-invariant basis of one-forms $e^1,\dotsc, e^6$ such that
\[de^1=de^4=e^{14},\quad de^2=de^5=e^{25},\quad  de^3=de^6=e^{36}.\]
We consider the $\SL(3,\R)$-structure for which this is an adapted coframe; it is parak\"ahler and Einstein with $s\neq0$. This Lie group is diffeomorphic to $\R^6$; in addition, it is easy to check that the space of leaves for both foliations is smooth and diffeomorphic to $\R^3$, so that Theorem~\ref{thm:aubinyau_dei_poveri} applies and the Bott-Chern class  $[d\lambda]$ is zero. Notice however that up to multiple, there is only one \emph{invariant} closed $(1,1)$-form: if a compatible Ricci-flat parak\"ahler metric exists, it cannot be invariant.
\end{example}

We conclude this section with a remark concerning the holonomy group of a parak\"ahler manifold. By construction, this is a subgroup of $\GL(n,\R)$. If it is also contained in $\GL_+(n,\R)$, the bundle $\Lambda^{n,0}$ is trivial and it is possible to fix a reduction to $\SL(n,\R)$. Ricci-flatness means that the \emph{restricted} holonomy is $\SL(n,\R)$; therefore, the full holonomy group has the form $G\times\GL(n,\R)$, with $G$ a discrete subgroup of $\GL_+(n,\R)$.
\begin{example}
Let $\lambda>0$ be a real constant and let 
\[\Omega=\{(x,y)\in\R^n\times \R^n \mid x_1>\frac1{1-\lambda}\}.\]
Let $g$ be the affine transformation
\[(x,y)\mapsto (\lambda x + e_1, \lambda^{-1}y).\]
The group generated by $g$ acts on $\Omega$ in a free, proper discontinuous way. Let $M=\Omega/\langle g\rangle$ be the quotient.

The flat parak\"ahler structure 
\[F=dx_1\wedge dy_1+\dots + dx_n\wedge dy_n, \quad h=\sum dx_i\odot dy_i\]
passes onto the quotient. It is easy to see that 
\[\{(\lambda^k dx_1, \dots, \lambda^k dx_n, \lambda^{-k}dy_1,\dots, \lambda^{-k}dy_n) \mid k\in\Z\}\]
defines a parallel $\Z$-structure on $\Omega$; this structure also passes onto the quotient. Moreover, taking parallel transport on $\Omega$ and observing that
\[\pi_{*gp}(\D{x_i})=\lambda^{-1}\pi_{*p}(\D{x_i}), \quad \pi_{*gp}(\D{y_i})=\lambda\pi_{*p}(\D{y_i}), \]
one sees that the holonomy of $M$ is precisely $\Z$. 
\end{example}

\section{Einstein examples}
\label{sec:examples}
In this section we show some examples of Einstein metrics associated to $\SL(n,\R)$-structures on Lie groups. With one exception (Example~\ref{example:non_nilpotent}), the Lie groups we consider are nilpotent with rational structure constants; therefore, each admits a compact quotient which carries an induced $\SL(n,\R)$-structure, also Einstein.

We will represent Lie groups by their structure constants; for instance the quadruplet $(0,0,0,12)$ represents a four-dimensional Lie group with a basis of left-invariant one-forms $e^1,e^2,e^3,e^4$ such that  
\[de^1=0,\ de^2=0,\ de^3=0,\ de^4=e^{12};\] 
the \emph{standard} $\SL(2,\R)$-structure on this Lie group is the one defined by the coframe $e^1,e^2,e^3,e^4$.

\begin{example}
\label{example:ricciflat}
A non-flat example of a parak\"ahler Ricci-flat manifold is the standard $\SL(n,\R)$-struc\-ture on the nilpotent Lie algebra $(24,0,0,0,0,35)$; in this case the curvature is
$-e^{34}\otimes e^{34}$.
\end{example}

\begin{example}
\label{example:non_nilpotent}
There also exist parak\"ahler Einstein Lie groups with non-zero scalar curvature: consider the one-parameter family of  Lie algebras
\[(14,25,36,t14,t25,t36), \quad t\in\R;\]
notice that these are not nilpotent (in fact, not even unimodular). The standard $\SL(n,\R)$-structure is Einstein, and Ricci-flat when $t=0$. In addition, the parak\"ahler metric is unique once one fixes the paracomplex structure, if one requires it to be invariant.
\end{example}

\begin{example}
The Lie group $(0,0,0,-\frac13 e^{23},-\frac13 e^{31},-\frac13 e^{12})$ is flat and nearly para\-k\"ahler; in this case, the minimal connection is the flat connection.

We can modify the example in such a way that the minimal connection is a non-trivial element in the subspace of $T^*\otimes\Sl(3,\R)$ isomorphic to $V_2$; imposing $d^2=0$, we obtain the following Lie algebra:
\begin{gather*}
de^1=-2  \lambda_3 e^{12}-2 e^{23}+2  e^{13} \lambda_2\\
de^2=-2  \lambda_3 e^{12} \lambda_2-2  e^{23} \lambda_2+2  e^{13} \lambda_2^{2}\\
de^3=-2  \lambda_3 e^{23}+2  \lambda_3 e^{13} \lambda_2-2  \lambda_3^{2} e^{12}\\
de^4={(e^{36}-e^{25})} \lambda_3 \lambda_2- \lambda_3 e^{24}-\frac{1}{3} e^{23}+ e^{34} \lambda_2- \lambda_3^{2} e^{26}+ e^{35} \lambda_2^{2}\\
de^5=\frac{1}{3} e^{13}-e^{34}+ \lambda_3 e^{15} \lambda_2- e^{35} \lambda_2+ \lambda_3^{2} e^{16}- {(e^{36}-e^{14})} \lambda_3\\
de^6=\lambda_3 e^{26}-\frac{1}{3} e^{12}+e^{24}- \lambda_3 e^{16} \lambda_2+ {(e^{25}-e^{14})} \lambda_2- e^{15} \lambda_2^{2}.
\end{gather*}
The standard structure on this Lie algebra is again flat and nearly parak\"ahler. A suitable change of basis, namely
\[\{-e^2+ e^1 \lambda_2,e^3- \lambda_3 e^1,-\frac{1}{6} e^1+e^4+ e^5 \lambda_2+ \lambda_3 e^6,\frac{1}{3} e^1+e^4+ e^5 \lambda_2+ \lambda_3 e^6,- e^6,- e^5\},\]
shows the Lie algebra to be nilpotent and isomorphic to $(0,0,0,12,14,24)$.
\end{example}

It follows from Theorem~\ref{thm:restateRicViaForms} that if the intrinsic torsion is purely in $W_2$, then the metric is Ricci flat. This leads us to a counterexample to the  paracomplex version of the Goldberg conjecture of almost-K\"ahler geometry. This ``para-Goldberg'' conjecture, as stated in \cite{Matsushita}, asserts that any compact almost parak\"ahler Einstein manifold is parak\"ahler.
\begin{proposition}
\label{prop:goldberg}
The standard $\SL(2,\R)$-structure on the nilmanifold $(0,0,0,12)$ defines a compact, almost parak\"ahler Ricci-flat manifold which is not parak\"ahler. Hence, the para-Goldberg conjecture does not hold.
\end{proposition}
\begin{proof}
One easily sees that $\lambda=0$ and the only component of the intrinsic torsion is $\tau_2$, implying Ricci-flatness (in fact, this metric is flat). The fact that $de^4$ has a component in $\Lambda^{2,0}$ shows that the paracomplex structure is not integrable.
\end{proof}

\begin{example}
Another example in the almost parak\"ahler class $W_2+W_6$ is the nilpotent Lie algebra
\[(0,0,46,0,12,0),\]
with the standard $\SL(3,\R)$-structure; its intrinsic torsion is
\[\tau_2 = e_2\otimes e^{12}, \quad \tau_6 = e_6\otimes e^{46}.\]
This metric is Ricci-flat, but the Levi-Civita curvature is  $e^{26}\otimes e^{26}$.
\end{example}

\begin{example}
A Ricci-flat example with intrinsic torsion in $W_3$ is given by the standard $\SL(n,\R)$-structure on the nilpotent Lie algebra $(0,0,0,0,0,45)$; this structure is almost parak\"ahler and flat, with
\[\tau = \frac{1}{2}  e_3\otimes e^{45}\in W_3.\]

A non-flat example is given by the standard $\SL(n,\R)$-structure on the $8$-dimensional nilmanifold $(0,0,0,0,0,0,56,57)$; then
\[\tau=\tau_3 = \frac{1}{2}  e_3\otimes e^{56}+\frac{1}{2}  e_4\otimes e^{57}\]
and the Ricci tensor is zero, but the curvature is $-\frac14 e^{56}\odot e^{45}$.
\end{example}

The nearly parak\"ahler examples we have shown so far are flat, and the geometry of flat nearly parak\"ahler structures is fully understood (see \cite{CortesSchafer}). We now illustrate a systematic approach to construct Ricci-flat  nearly parak\"ahler metrics that are not flat. Our examples have intrinsic torsion in $W_1$; by Theorem~\ref{thm:restateRicViaForms}, this condition ensures Ricci-flatness.

The intrinsic torsion of an almost paracomplex structure can be identified with the Nijenhuis tensor, which has two components
\[N^H+N^V \in \Lambda^2V^*\otimes H+\Lambda^2H^*\otimes V.\]
Under the structure group $\GL(n,\R)$, these two irreducible modules decompose respectively into $W_1+W_2$ and $W_5+W_6$; by Proposition~\ref{prop:paracomplexit}, we can write
\[N^H=4\partial(\tau_1+\tau_2), \quad N^V=4\partial(\tau_5+\tau_6).\]
The key observation is that $N^H$ can be identified with a linear bundle map that does not depend on the almost  paracomplex structure $K$ but only on the $1$-eigendistribution $\mathcal{V}$, namely
\[\tilde N^H\colon  \mathcal{V}\ann\to \Lambda^2\mathcal{V}^*,\]
where $\mathcal{V}\ann$ is the subbundle of $T^*M$ whose fibre at $p$ is the annihilator of $\mathcal{V}_p$.

\begin{proposition}
\label{prop:almostW1}
Let $\lie{g}$ be a Lie algebra of dimension $2n$; let $V\subset\lie{g}$ be a subspace of dimension $n$ such that:
\begin{enumerate}
\item $\tilde N^H$ has rank three;
\item the image $W=\im \tilde N^H$ consists of simple forms, i.e. if $\eta\in W$ then $\eta\wedge\eta=0$;
\item there is no $\sigma\in V^*\setminus\{0\}$ such that $\sigma\wedge\eta=0$ for all $\eta\in W$.
\end{enumerate}
Then for any subalgebra $H\subset \lie{g}$ complementary to $V$ there is a $\GL(n,\R)$-structure compatible with the splitting $\lie{g}=V\oplus H$
such that
\[\tau_2=0=\tau_5=\tau_6.\]
\end{proposition}
\begin{proof}
Conditions 1 and 2 imply that $V^*$ is spanned by $e^{1},\dots, e^{n}$ such that $W$ is either of the form
\[\Span{e^{12},e^{13},e^{23}}\]
or 
\[\Span{e^{14},e^{24},e^{34}};\]
the latter possibility is ruled out by Condition 3.

Having fixed the splitting $V\oplus H$, we can identify $V\ann$ with $H^*$; in particular, we can write 
\[\tilde N^H=\eta^1\otimes e^{23}+\eta^2\otimes e^{31}+\eta^3\otimes e^{12},\]
where $\eta^1,\eta^2,\eta^3$ are linearly independent elements of $H^{**}$. Let $e^{n+1},\dots, e^{2n}$ be a basis of $H^*$ such that $\eta^i(e^{n+j})=\delta_{ij}$; relative to the coframe $e^1,\dotsc, e^{2n}$, we can write 
\[N^H=e^{23}\otimes e_{n+1}+e^{31}\otimes e_{n+2}+e^{12}\otimes e_{n+3};\]
thus, this coframe defines an $\SL(n,\R)$-structure such that 
\[\tau_1+\tau_2 =\frac18( e^1\otimes e^{23}+e^2\otimes e^{31}+e^3\otimes e^{12}),\] i.e. $\tau_2=0$. The fact that $\tau_5=0=\tau_6$ follows from $H$ being integrable.
\end{proof}
A similar result holds for distributions $\mathcal{V}$ and $\mathcal{H}$ on a manifold, if Conditions 1---3 are assumed to hold at each point, but we will only consider the case of Lie groups. Then, the condition $\tau_2=0$ can be expressed as a linear equation in $F$ by using the following contraction:
\begin{equation*}
c\colon (\Lambda^2V^*\otimes H)\otimes \Lambda^2T^*\to \Lambda^2V^*\otimes V^*, \quad (\eta\otimes h)\otimes F\mapsto \eta\otimes h\hook F.
\end{equation*}

\begin{corollary}
Let $\lie{g}=V\oplus H$ be a splitting such as in Proposition~\ref{prop:almostW1}, and assume that $d(\Lambda^{0,n})\subset\Lambda^{2,n-1}$. Then any non-degenerate form $F\in\Lambda^{1,1}$ such that 
\begin{equation}
 \label{eqn:linearricciW1}
dF\in\Lambda^{3,0}, \quad c(N^H,F)\in\Lambda^{3,0}
\end{equation}
defines a Ricci-flat $\GL(n,\R)$-structure with intrinsic torsion in $W_1$.
\end{corollary}
\begin{proof}
Fix a compatible reduction to $\SL(n,\R)$; then Proposition~\ref{prop:forms} implies
\[\lambda=0, \quad \tau_3=0=\tau_4=\tau_7=\tau_8.\]
By Proposition~\ref{prop:almostW1}, $\tau_2$, $\tau_5$ and $\tau_6$ vanish as well; thus the intrinsic torsion is in $W_1$, and by Theorem~\ref{thm:restateRicViaForms} the metric is Ricci-flat.
\end{proof}

This gives us an effective recipe to find Ricci-flat nearly parak\"ahler examples: for a given  Lie algebra $\lie{g}$ of dimension $2n$,  we seek a subspace $V\subset\lie{g}$ of dimension $n$ that satisfies the conditions of Proposition~\ref{prop:almostW1}, and then a complementary integrable distribution $H$ such that $d(\Lambda^{0,n})\subset\Lambda^{2,n-1}$. At this point it is only a matter of computing the space of forms of type $(1,1)$ that satisfy the linear equations \eqref{eqn:linearricciW1}, and verifying whether it contains a non-degenerate form. Implementing this strategy with a computer on Lie algebras of the form $\lie{h}\oplus\R$, where $\lie{h}$ ranges among $7$-dimensional nilpotent Lie algebras as classified by \cite{Gong}, gives the following:
\begin{theorem}
\label{thm:nkricciflat}
Each Lie algebra in Table~\ref{fig:nkricciflat} admits a family of non-flat, Ricci-flat nearly-parak\"ahler structures with intrinsic torsion $\tau_1\neq0$ and non-zero curvature,  depending on parameters  $\lambda,\mu\neq0$ and $k$.
\end{theorem}
We emphasize that this is not a classification; whilst our program did not find any other example, we do not know whether there exist other $8$-dimensional nilpotent Lie algebras with a strictly nearly parak\"ahler structure which is Ricci-flat but not flat.

\begin{table}[t]
\caption{Ricci-flat nearly-parak\"ahler structures on nilpotent Lie algebras}\label{fig:nkricciflat}
\resizebox{\textwidth}{!}{
\bgroup
\def\arraystretch{1.5}
\begin{tabular}{>{$}l<{$}|>{$}l<{$}}
\lie{h} & \lie{g} \text{ (in adapted coframe)}\\
\midrule
0,0,0,e^{12},e^{23}+e^{14},e^{24},e^{15}-e^{34}&0,0,-\lambda e^{18},0,-\mu e^{23} +\lambda e^{78}, \mu e^{13},- \mu e^{12} +\lambda\mu  e^{38},0\\
0,0,0,e^{12},e^{23}+e^{14},e^{24}+e^{13},-e^{34}+e^{15}&0,0,-\lambda e^{18},0,\lambda e^{78}- \mu e^{23},-\lambda\mu e^{28} + \mu e^{13},\lambda \mu  e^{38} - \mu e^{12},0\\
0,0,e^{12},0,e^{24}+e^{13},e^{23},e^{34}+e^{15}+e^{26}&0,-\lambda e^{18},0,0, \mu e^{23}+e^{17}+\lambda e^{68}-k\lambda  e^{18} ,- \mu e^{13}+\lambda\mu e^{28}, \mu e^{12},0\\
0,0,0,e^{12},e^{23}+e^{14},-e^{24}+e^{13},-e^{34}+e^{15}&0,0,-\lambda e^{18},0,- \mu e^{23} +\lambda e^{78}, \mu e^{13} +\lambda \mu e^{28},- \mu e^{12}+\lambda\mu  e^{38} ,0\\
0,0,e^{12},0,e^{13}+e^{24},e^{23},-e^{26}+e^{15}+e^{34}&0,-\lambda e^{18},0,0,-e^{17}+ \mu e^{23} +\lambda e^{68}+k\lambda  e^{18},\lambda\mu  e^{28} - \mu e^{13}, \mu e^{12},0
\end{tabular}
\egroup
}
\end{table}

\section{Non-existence results}
The examples of Section~\ref{sec:examples} correspond to left-invariant structures on nilpotent Lie groups; in this section we give some non-existence results in the same context.

\begin{lemma}
\label{lemma:paracomplexNilpotent}
Every left-invariant paracomplex $\SL(n,\R)$-structure on a nilpotent Lie group satisfies
\[\lambda=\frac{n-1}n (f_8-f_4).\]
\end{lemma}
\begin{proof}
Let $\lie{g}$ be a nilpotent Lie algebra with a fixed $\SL(n,\R)$-structure. For any $X\in\lie{g}$, $\ad X$ is nilpotent; the induced action on $\Lambda\lie{g}^*$ is also nilpotent. 

By Proposition~\ref{prop:forms}, 
\[d\alpha = (-n\lambda-(n-1)f_4)\wedge \alpha;\]
whenever $X$ is in $H$, $\alpha$ is an eigenvector for the nilpotent endomorphism $\ad X$ with eigenvalue \[(-n\lambda-(n-1)f_4)(X);\] it follows that
\[n\lambda^{0,1}=-(n-1)f_4.\]
The same argument applied to $\beta$ yields $n\lambda^{1,0}=(n-1)f_8$.
\end{proof}

Theorem~\ref{thm:restateRicViaForms} immediately implies:
\begin{corollary}
Invariant parak\"ahler structures on nilpotent Lie groups are Ricci-flat.
\end{corollary}
In contrast with the Riemannian case, Ricci-flat nilpotent Lie groups are not necessarily abelian; Section~\ref{sec:examples} contains several examples. Alongside the parak\"ahler case (Example~\ref{example:ricciflat}), we have constructed Ricci-flat examples with intrinsic torsion in $W_1$, $W_2$ and $W_3$; the absence of the class $W_4$ can be explained by the following:
\begin{proposition}
\label{prop:nilpotentw4w8}
A nilpotent Lie group with an invariant, Einstein almost parak\"ahler structure with intrinsic torsion in $W_4+W_8$
is only Einstein if it is parak\"ahler, in which case it is Ricci-flat.
\end{proposition}
\begin{proof}
Fix a reduction to $\SL(n,\R)$; by Lemma~\ref{lemma:paracomplexNilpotent},
\[d\lambda = \frac{n-1}n (df_8-df_4).\]
On the other hand, Proposition~\ref{prop:forms} gives
$dF=-2(f_4+f_8)\wedge F;$
taking the exterior derivative, 
$0=-2(df_4+df_8)\wedge F,$
and consequently $df_8=-df_4$. In addition, this two-form has type $(1,1)$ because the structure is paracomplex. Thus,
\[d\lambda = -2\frac{n-1}n df_4.\]
The formula for $\ric'$ gives
\[\ric'=2(n-2)df_4+nd\lambda-2(n-1)f_4\wedge f_8 \mod F\]
Thus, the metric is Einstein if there exists $k$ such that
\begin{equation}
 \label{eqn:f4f8einstein}
df_4+(n-1)f_4\wedge f_8 =kF;
\end{equation}
taking the exterior derivative,
\begin{multline*}
(n-1)df_4\wedge f_8 +(n-1)f_4\wedge df_4=-2k(f_4+f_8)\wedge F \\
= -2(f_4+f_8)\wedge(df_4+(n-1)f_4\wedge f_8) =-2(f_4+f_8)\wedge df_4;
 \end{multline*}
this gives
\[(n+1)df_4\wedge(f_4+f_8)=0,\]
and decomposing according to type
\[df_4\wedge f_4=0, \quad df_4\wedge f_8=0.\]
Therefore,
\[df_4\in \Span{f_4\wedge f_8};\]
this is only possible if $df_4$ is zero; by \eqref{eqn:f4f8einstein} this implies $f_4=f_8=0$, which in turn gives $\lambda=0$.
\end{proof}
\begin{remark}
Given an $\SL(n,\R)$-structure on a manifold, it follows from Proposition~\ref{prop:forms} that the intrinsic torsion is contained in $W_4+W_8+W^{1,0}+W^{0,1}$ if and only if $N=0$ and $dF=\theta\wedge F$ for some $1$-form $\theta$; such structures are known in the literature as locally conformally parak\"ahler \cite{GadeaMasque}. The result of Proposition~\ref{prop:nilpotentw4w8} can be rephrased by saying that an Einstein, left-invariant locally conformally parak\"ahler structure on a nilpotent Lie group is parak\"ahler.
\end{remark}

\section{Composite intrinsic torsion classes}
In this section we give a simple construction to build almost parahermitian manifolds in a prescribed intrinsic torsion class using almost parahermitian manifolds of lower dimension as building blocks.

We shall denote by $T_n$ the $2n$-dimensional $\GL(n,\R)$-module \eqref{eqn:linearmodel}; let $W_i(T_{n})$ denote the corresponding intrinsic torsion spaces. We will also use a subscript $n$ to denote the $2n$-dimensional forms $F_n$, $\alpha_n$ and $\beta_n$.

We identify  $T_n\oplus T_m$  with $T_{n+m}$ by the isomorphism that maps the basis
\[\{e^1,\dots,e^n,\allowbreak f^1,\dots, f^m,\allowbreak e^{n+1},\dots,e^{2n},\allowbreak f^{m+1},\dots, f^{2m}\}\]
onto the basis
\[\{E^1,\dotsc, E^{2(m+n)}\};\]
here, $\{e_i\}$, $\{f_i\}$ and $\{E_i\}$ are the standard bases of $T_n$, $T_m$ and $T_{n+m}$ respectively. We obtain the following relations between the forms $F$, $\alpha$ and $\beta$:
\[F_{n+m}=F_n+F_m,\quad \alpha_{n+m}=\alpha_{n}\wedge\alpha_{m},\quad\beta_{n+m}=\beta_{n}\wedge\beta_{m}.\]
Moreover, we have the following lemma:
\begin{lemma}\label{lemma:higherdim}
 The following relations hold:
\begin{gather*}
W_i(T_{n})\subset W_i(T_{n+m})\qquad\text{for}\ i\neq4,8,\\
W_4(T_{n})\subset W_3(T_{n+m})\oplus W_4(T_{n+m}),\qquad W_8(T_{n})\subset W_7(T_{n+m})\oplus W_8(T_{n+m}).
\end{gather*}
Moreover, the projections
\begin{align*}
W_3(T_n)\oplus W_4(T_n)\oplus W_3(T_m)\oplus W_4(T_m)&\to W_3(T_{n+m})\\
W_7(T_n)\oplus W_8(T_n)\oplus W_7(T_m)\oplus W_8(T_m)&\to W_7(T_{n+m})
\end{align*}
are injective.
\end{lemma}
\begin{proof}
Consider the  commutative diagram
\begin{equation*}
\xymatrix{
\Lambda^2T^*_n\otimes T_n\ar@{->}[r]^f\ar@{->}[d] & \Lambda^2T^*_{n+m}\otimes T_{n+m}\ar@{->}[d] \\
\text{coker}(\partial_n)\ar@{->}[r]^{\hat{f}} & \text{coker}(\partial_{n+m})
}
\end{equation*}
where $f(a_{IJK}e^{IJ}\otimes e_K)=a_{IJK}e^{IJ}\otimes e_K$. The map $\hat{f}$ is injective because $f(\Lambda^2T^*_n\otimes T_n)\cap\partial(T^*_{n+m}\otimes \gl({n+m}))=\{0\}$.

For $i\neq4,8$, let $\tau_i$  denote an element of $W_i(T_n)$. It is obvious that $\Lambda^3V_n^*\subset\Lambda^3V_{n+m}^*$; hence, $\hat{f}(\tau_1)\in W_1(T_{n+m})$. Then the inclusion easily follows for $\tau_2$: in fact, $\hat{f}(\tau_1+\tau_2)\in V^*_{n+m}\otimes\Lambda^2V^*_{n+m}$ and  $\hat{f}(\tau_2)$ has no component in $\Lambda^3V_{n+m}^*$ because  $\tau_2$ has no component in $\Lambda^3V_{n}^*$. Now recall that $W_3(T_n)$ is the kernel of the contraction $V_n^*\otimes\Lambda^2V_n\to V_n$;  it easily follows that $\hat{f}(\tau_3)$ belongs to the kernel of the corresponding contraction defined on $V^*_{n+m}\otimes\Lambda^2V_{n+m}$.

Finally, if $\tau_4=a_ie^k\otimes e_{ki}$, then $\hat{f}(\tau_4)\in V^*_{n+m}\otimes\Lambda^2V_{n+m}$ is given by
\[
\hat{f}(\tau_4)=a_ie^k\otimes e_{ki}
=\biggl(a_ie^k\otimes e_{ki}-\frac{n-1}{m}a_if^h\otimes f_h\wedge e_i\biggr)+\frac{n-1}{m}a_if^h\otimes f_h\wedge e_i,
\]
where the first summand lies in $W_3(T_{n+m})$ and the second one lies in $W_4(T_{n+m})$; a similar decomposition applies to any element $\tau_4'= a_i'f^k\otimes f_{ki}$ of $W_4(T_{n+m})$. Writing explicitly the image in $W_3(T_{n+m})$ of an element $\tau_3+\tau_4+\tau_3'+\tau_4'$ and imposing that the component in $V_m^*\otimes (V_m\wedge V_n)$ is zero we obtain $\tau_4=0$; similarly, one finds that $\tau_4'=0$. Injectivity follows.

The same arguments can be applied to $\tau_5, \tau_6,\tau_7$ and $\tau_8$. 
\end{proof}

At the geometric level, given two almost parahermitian manifolds $(N,g_n,\allowbreak K_n,F_n)$ and $(M,g_m,K_m,F_m)$, on the product $N\times M$ we consider the natural almost parahermitian structure
\[(N\times M, g_n+g_m, K_n+K_m,F_n+F_m).\]
In our setting, an \emph{intrinsic torsion class} is a subset of $\{\mathcal{W}_1,\dots, \mathcal{W}_8\}$; it is customary to represent a subset $\{\mathcal{W}_{i_1},\dots, \mathcal{W}_{i_h}\}$ as a formal sum 
\begin{equation}
 \label{eqn:formalsum}
\mathcal{W}_{i_1}+ \dots + \mathcal{W}_{i_h}.
\end{equation}
Accordingly, the union of two intrinsic torsion classes $I$ and $J$ is written as $I+J$.

An almost parahermitian manifold $(M,g,K,F)$ is in the intrinsic torsion class \eqref{eqn:formalsum} if the components of the intrinsic torsion which are not identically zero are precisely $\tau_{i_1},\dots, \tau_{i_h}$; we write $\mathcal{W}(M)=\mathcal{W}_{i_1}+ \dots + \mathcal{W}_{i_h}$.
\begin{proposition}\label{prop:higherdim}
If $(N,g,K,F)$ is a $2n$-dimensional almost parahermitian manifold with $\tau_4=0=\tau_8$, then for each $m\in\N$ the natural almost parahermitian structure on the $2(m+n)$-dimensional manifold $N\times\R^{2m}$ is in the same intrinsic torsion class.
\end{proposition}
\begin{proof}
Follows from Proposition~\ref{prop:forms} and Lemma \ref{lemma:higherdim}.
\end{proof}
More generally, we can find non-trivial intrinsic torsion classes by combining two almost parahermitian manifolds. Indeed, by the same arguments we can prove the following:
\begin{proposition}

Let $M,N$ be two almost parahermitian manifolds of dimension respectively $2m$ and $2n$, each with $\tau_4=0=\tau_8$; then the intrinsic torsion class of the natural almost parahermitian structure on $M\times N$ is the union of the intrinsic torsion classes of $M$ and $N$.
\end{proposition}
The general case requires more notation. We say a map from the set of intrinsic torsion classes to itself is additive if it satisfies $g(I + J)=g(I) + g(J)$. Consider the additive map
\[g\colon \mathcal{W}_i\to \begin{cases} \mathcal{W}_i & i\neq 4,8\\ \mathcal{W}_3+\mathcal{W}_4 & i=4\\\mathcal{W}_7+\mathcal{W}_8 & i=8   \end{cases};\]
again by Lemma \ref{lemma:higherdim}, we find:
\begin{proposition}
\label{prop:products}
The intrinsic torsion of the natural almost parahermitian structure on a product $N\times M$ is
\[\mathcal{W}(N\times M)=g(\mathcal{W}(N)+\mathcal{W}(M)).\]
\end{proposition}
In the situation of Proposition~\ref{prop:products}, if $M$ and $N$ are Einstein with the same scalar curvature $s$, then the product $M\times N$ is also Einstein with scalar curvature $s$. In particular, the results of Section~\ref{sec:examples} imply:
\begin{proposition}
\label{prop:ricciflatproducts}
Each intrinsic torsion class involving only $\mathcal{W}_1$, $\mathcal{W}_2$, $\mathcal{W}_3$, $\mathcal{W}_5$, $\mathcal{W}_6$ and $\mathcal{W}_7$ contains compact manifolds with a Ricci-flat $\GL(n,\R)$-structure for $n\gg1$.
\end{proposition}

Unlike $\mathcal{W}_1$ and $\mathcal{W}_2$, the intrinsic torsion classes $\mathcal{W}_3$ and $\mathcal{W}_4$ are not Ricci-flat; this is indicated by the formulae of Theorem~\ref{thm:restateRicViaForms}, and can be verified through the following examples:
\begin{example}
In the class $\mathcal{W}_3$, the standard $\SL(3,\R)$-structure on the Lie algebra $(0,0,0,0,45,46)$ has intrinsic torsion
\[\tau_3 = \frac{1}{2}  e_2\otimes e^{45}-\frac{1}{2}  e_3\otimes e^{46};\]
the Ricci tensor is $-e^4\otimes e^4$.
\end{example}
\begin{example}
A non-Einstein example with intrinsic torsion contained in $W_4$ is the standard $\SL(3,\R)$-structure on the Lie algebra \[(-14,0,0, 0,45,46).\] In this case the Ricci curvature is $e^4\otimes e^4\in S^2V$.
\end{example}

\begin{remark}
Using this last example, together with those of Section \ref{sec:examples}, the construction of Proposition~\ref{prop:products} allows one to produce almost parahermitian structures in any given intrinsic torsion class.
\end{remark}

\bibliographystyle{plain}
\bibliography{parahermitian}

\def\cprime{$'$}
\begin{thebibliography}{10}

\bibitem{AlekseevskiKimelFel}
D.~V. Alekseevsky and B.~N. Kimel{\cprime}fel{\cprime}d.
\newblock Structure of homogeneous {R}iemannian spaces with zero {R}icci
  curvature.
\newblock {\em Funkcional. Anal. i Prilo\v Zen.}, 9(2):5--11, 1975.

\bibitem{AlekseevskyMedoriTomassini:Homogeneous2009}
D.~V. Alekseevsky, C.~Medori, and A.~Tomassini.
\newblock Homogeneous para-{K}\"ahlerian {E}instein manifolds.
\newblock {\em Uspekhi Mat. Nauk}, 64(1(385)):3--50, 2009.

\bibitem{AngellaRossi:cohomology2012}
D.~Angella and F.~A. Rossi.
\newblock Cohomology of {$\mathbf{D}$}-complex manifolds.
\newblock {\em Differential Geom. Appl.}, 30(5):530--547, 2012.

\bibitem{Aubin}
T.~Aubin.
\newblock \'{E}quations du type {M}onge-{A}mp\`ere sur les vari\'et\'es
  k\"ahl\'eriennes compactes.
\newblock {\em Bull. Sci. Math. (2)}, 102(1):63--95, 1978.

\bibitem{BedulliVezzoni:SU3}
L.~Bedulli and L.~Vezzoni.
\newblock The {R}icci tensor of {SU}(3)-manifolds.
\newblock {\em J. Geom. Phys.}, 57(4):1125--1146, 2007.

\bibitem{Bryant}
R.~L. Bryant.
\newblock Metrics with exceptional holonomy.
\newblock {\em Annals of Mathematics}, 126:525--576, 1987.

\bibitem{Bryant:remarks}
R.~L. Bryant.
\newblock Some remarks on {$G_2$}-structures.
\newblock In {\em Proceedings of {G}\"okova {G}eometry-{T}opology {C}onference
  2005}, pages 75--109. G\"okova Geometry/Topology Conference (GGT), G\"okova,
  2006.

\bibitem{ChiossiSalamon}
S.~Chiossi and S.~Salamon.
\newblock The intrinsic torsion of {$SU(3)$} and {$G_2$ structures}.
\newblock In {\em Differential Geometry, Valencia 2001}, pages 115--133. World
  Scientific, 2002.

\bibitem{ChursinSchaferSmoczyk}
M.~Chursin, L.~Sch{\"a}fer, and K.~Smoczyk.
\newblock Mean curvature flow of space-like {L}agrangian submanifolds in almost
  para-{K}\"ahler manifolds.
\newblock {\em Calc. Var. Partial Differential Equations}, 41(1-2):111--125,
  2011.

\bibitem{Conti:qc}
D.~Conti.
\newblock Intrinsic torsion in quaternionic contact geometry.
\newblock {\em Ann. Sc. Norm. Super. Pisa Cl. Sci. (5)}, 16(2):625--674, 2016.

\bibitem{CortesLeistner}
V.~Cort\'es, T.~Leistner, L.~Sch\"afer, and F.~Schulte-Hengesbach.
\newblock Half-flat structures and special holonomy.
\newblock {\em Proc. Lond. Math. Soc. (3)}, 102(1):113--158, 2011.

\bibitem{CMMS:VectorI2004}
V.~Cort{\'e}s, C.~Mayer, T.~Mohaupt, and F.~Saueressig.
\newblock Special geometry of {E}uclidean supersymmetry. {I}. {V}ector
  multiplets.
\newblock {\em J. High Energy Phys.}, 2004(3):028, 73 pp. (electronic), 2004.

\bibitem{CortesSchafer}
V.~Cort{\'e}s and L.~Sch{\"a}fer.
\newblock Geometric structures on {L}ie groups with flat bi-invariant metric.
\newblock {\em J. Lie Theory}, 19(2):423--437, 2009.

\bibitem{CruceanuFortunyGadea:survey1996}
V.~Cruceanu, P.~Fortuny, and P.~M. Gadea.
\newblock A survey on paracomplex geometry.
\newblock {\em Rocky Mountain J. Math.}, 26(1):83--115, 1996.

\bibitem{EtayoSantamaria}
F.~Etayo, R.~Santamar{\'{\i}}a, and U.~R. Tr{\'{\i}}as.
\newblock The geometry of a bi-{L}agrangian manifold.
\newblock {\em Differential Geom. Appl.}, 24(1):33--59, 2006.

\bibitem{FultonHarris}
W.~Fulton and J.~Harris.
\newblock {\em Representation theory}, volume 129 of {\em Graduate Texts in
  Mathematics}.
\newblock Springer-Verlag, New York, 1991.

\bibitem{GadeaMasque}
P.~M. Gadea and J.~M. Masque.
\newblock Classification of almost para-{H}ermitian manifolds.
\newblock {\em Rend. Mat. Appl. (7)}, 11(2):377--396, 1991.

\bibitem{Gong}
M.-P. Gong.
\newblock {\em Classification of nilpotent {L}ie algebras of dimension 7 (over
  algebraically closed fields and {R})}.
\newblock ProQuest LLC, Ann Arbor, MI, 1998.
\newblock Thesis (Ph.D.)--University of Waterloo (Canada).

\bibitem{GrayHervella}
A.~Gray and L.~Hervella.
\newblock The sixteen classes of almost {H}ermitian manifolds.
\newblock {\em Ann. Mat. Pura e Appl.}, 282:1--21, 1980.

\bibitem{HarveyLawson}
F.~R. Harvey and Jr. H.~B. Lawson.
\newblock Split special {L}agrangian geometry.
\newblock In {\em Metric and differential geometry}, volume 297 of {\em Progr.
  Math.}, pages 43--89. Birkh\"auser/Springer, Basel, 2012.

\bibitem{IvanovZamkovoy:parahermitian}
S.~Ivanov and S.~Zamkovoy.
\newblock Parahermitian and paraquaternionic manifolds.
\newblock {\em Differential Geom. Appl.}, 23(2):205--234, 2005.

\bibitem{KimMcCannWarren}
Y.-H. Kim, R.~J. McCann, and M.~Warren.
\newblock Pseudo-{R}iemannian geometry calibrates optimal transportation.
\newblock {\em Math. Res. Lett.}, 17(6):1183--1197, 2010.

\bibitem{Libermann:paracomplex1952}
P.~Libermann.
\newblock Sur les vari\'et\'es presque paracomplexes.
\newblock In {\em Colloque de topologie et g\'eom\'etrie diff\'erentielle,
  {S}trasbourg, 1952, no. 5}, page~10. La Biblioth\`eque Nationale et
  Universitaire de Strasbourg, 1953.

\bibitem{CabreraSwann}
F.~Mart{\'{\i}}n~Cabrera and A.~Swann.
\newblock Curvature of almost quaternion-{H}ermitian manifolds.
\newblock {\em Forum Math.}, 22(1):21--52, 2010.

\bibitem{Matsushita}
Y.~Matsushita.
\newblock Counterexamples of compact type to the {G}oldberg conjecture and
  various version of the conjecture.
\newblock In {\em Topics in contemporary differential geometry, complex
  analysis and mathematical physics}, pages 222--233. World Sci. Publ.,
  Hackensack, NJ, 2007.

\bibitem{RossiPhD}
F.~A. Rossi.
\newblock {\em \textbf{D}-complex structures: cohomological properties and
  deformations}.
\newblock PhD thesis, Universit\`{a} degli Studi di Milano - Bicocca, 2013.
\newblock \texttt{http://hdl.handle.net/10281/41976}.

\bibitem{Schafer:conical}
L.~Sch{\"a}fer.
\newblock Conical {R}icci-flat nearly para-{K}\"ahler manifolds.
\newblock {\em Ann. Global Anal. Geom.}, 45(1):11--24, 2014.

\bibitem{SchaferSchulteHengesbach}
L.~Sch\"afer and F.~Schulte-Hengesbach.
\newblock Nearly pseudo-{K}\"ahler and nearly para-{K}\"ahler six-manifolds.
\newblock In {\em Handbook of pseudo-{R}iemannian geometry and supersymmetry},
  volume~16 of {\em IRMA Lect. Math. Theor. Phys.}, pages 425--453. Eur. Math.
  Soc., Z\"urich, 2010.

\bibitem{Sekigawa}
K.~Sekigawa.
\newblock On some compact {E}instein almost {K}\"ahler manifolds.
\newblock {\em J. Math. Soc. Japan}, 39(4):677--684, 1987.

\bibitem{Steenrod}
N.~Steenrod.
\newblock {\em The {T}opology of {F}ibre {B}undles}.
\newblock Princeton Mathematical Series, vol. 14. Princeton University Press,
  Princeton, N. J., 1951.
\newblock Reprinted in Princeton Landmarks in Mathematics, Princeton University
  Press, Princeton, NJ, (1999).

\bibitem{Yau}
S.-T. Yau.
\newblock On the {R}icci curvature of a compact {K}\"ahler manifold and the
  complex {M}onge-{A}mp\`ere equation. {I}.
\newblock {\em Comm. Pure Appl. Math.}, 31(3):339--411, 1978.

\end{thebibliography}

\small\noindent Dipartimento di Matematica e Applicazioni, Universit\`a di Milano Bicocca, via Cozzi 55, 20125 Milano, Italy.\\
\texttt{diego.conti@unimib.it}\\
\texttt{federico.rossi@unimib.it}

\end{document}